\documentclass{amsart}
\usepackage[utf8]{inputenc}

\usepackage{amssymb}
\usepackage{tikz}
\usepackage{tikz-cd}
\usetikzlibrary{matrix,arrows,calc,decorations.pathmorphing,decorations.pathreplacing,positioning}
\usepackage{graphicx}
\usepackage{caption}
\usepackage{url}
\usepackage{color}
\usepackage[all]{xy}
\SelectTips{cm}{}

\DeclareMathOperator{\Br}{\mathsf{Br}}
\DeclareMathOperator{\RBr}{\mathsf{RBr}}
\DeclareMathOperator{\URBr}{\mathsf{uRBr}}

\DeclareMathOperator{\id}{\mathsf{id}}
\DeclareMathOperator{\G}{\mathsf{G}}

\def\Ofrak{\mathfrak{O}}

\def\C{\mathsf{D}}
\def\fr{\mathrm{fr}}

\def\OO{\mathsf{O}}
\def\EE{\mathsf{E}}
\def\lra{\longrightarrow}

\def\op{\mathrm{op}}
\def\surj{\mathrm{surj}}
\def\inj{\mathrm{inj}}

\def\Set{\mathrm{Set}}
\def\I{\mathbb{I}}

 \newtheorem{thm}{Theorem}[section]
  \newtheorem{taggedtheoremx}{Theorem}
 \newenvironment{taggedtheorem}[1]
 {\renewcommand\thetaggedtheoremx{#1}\taggedtheoremx}
 {\endtaggedtheoremx}
 \newtheorem{cor}[thm]{Corollary}
 \newtheorem{lem}[thm]{Lemma}
 \newtheorem{prop}[thm]{Proposition}
 
 \theoremstyle{definition}
 \newtheorem{defn}[thm]{Definition}
   \newtheorem{rem}[thm]{Remark}
    \newtheorem{notat}[thm]{Notation}
   \newtheorem{examp}[thm]{Example}

\newcommand{\brbinom}[2]{\genfrac{[}{]}{0pt}{}{#1}{#2}}

\numberwithin{equation}{section}

\begin{document}


\title{Cloning systems and action operads}

\author[J. Aramayona]{Javier Aramayona}
\address{Javier Aramayona: Instituto de Ciencias Matem\'aticas, ICMAT (CSIC-UAM-UC3M-UCM). Nicol\'as Cabrera, 13--15. 28049, Madrid, Spain}
\email{javier.aramayona@icmat.es}
\thanks{J.A. was supported by grant  PGC2018-101179-B-I00 and acknowledges
support from the Spanish Ministry of Science, Innovation, and Universities, through the Severo
Ochoa Programme for Centres of Excellence in R\&D (CEX2019-000904-S and CEX2023-001347-S).  F.C. was supported by grant SI3/PJI/2021-00505 from Comunidad de Madrid and PID2022-142024NB-I00. V.C. was partially supported by grants PID2020- 117971GB-C21 funded by MCIN/AEI/10.13039/501100011033, US-1263032 (US/JUNTA/FEDER, UE), P20 01109 (JUNTA/FEDER, UE)  and FPU17/01871. J.J.G. was supported by 
grants PID2020-117971GB-C22 and CEX2020-001084-M funded by MCIN/AEI/10.13039/501100011033 and grant 2021-SGR-00697 funded by the Catalan Government.}
\author[F. Cantero]{Federico Cantero}
\address{Federico Cantero: Departamento de Matemáticas, Universidad Autónoma de Madrid \& ICMAT. Calle Francisco Tomás y Valiente, 7. 28049, Madrid, Spain}
\email{federico.cantero@uam.es}
\author[V. Carmona]{V\'ictor Carmona} 
\address{Víctor Carmona: Max Planck Institute for Mathematics in the Sciences. Inselstrasse 22, 04103, Leipzig, Germany}
\email{vcarmonamath@gmail.com}
\author[J.J. Gutiérrez]{Javier J. Guti\'errez}
\address{Javier J. Guti\'errez: Departament de Matem\`atiques i Inform\`atica, Universitat de Barcelona (UB). Gran Via de les Corts Catalanes 585. 08007, Barcelona, Spain}
\email{javier.gutierrez@ub.edu}
\date{\today}

\begin{abstract}
Action operads and cloning systems are, respectively, the main ingredients in Thumann's and Witzel--Zaremsky's approaches for axiomatically constructing Thompson-like groups. 
In this paper, we prove that action operads are equivalent to cloning systems that admit a certain extra structure, and which we call {\em restricted operadic} cloning systems. In addition, we describe their relation with crossed interval groups and product categories. 
\end{abstract}

\maketitle

\section{Introduction}
The umbrella term {\em Thompson-like groups} makes reference to a vast family of groups that are in some way reminiscent of one of the classical groups $F$, $T$ and $V$ of R. Thompson \cite{CFP}. Apart from these three groups, prominent examples of Thompson-like groups include Stein's groups of PL homeomorphisms \cite{Ste}; Guba--Sapir's  diagram groups \cite{GuSa}; Brin's higher-dimensional Thompson groups \cite{Bri04}; Belk--Forrest's rearrangement groups of fractals \cite{BeFo}; the braided Thompson group of Brin \cite{Bri07} and Dehornoy \cite{Deh}; Wahl's  ribbon Thompson group \cite{Wahl}; the  {\em asymptotic mapping class groups} of surfaces and higher-dimensional manifolds \cite{FK04, FK08,FK09,AF20,ABF+21,GLU20}, etc.

As may be appreciated from the above list of examples, Thompson-like groups arise in a variety of different ways.
With this motivation, the independent results of Witzel--Zaremsky \cite{WZ} and Thumann \cite{Th} offer unified frameworks for constructing Thompson-like groups. Witzel-Zaremsky achieve this in terms of {\em cloning systems}, which provide a recipe for ``twisting'' a direct limit of groups into a Thompson-like group. In turn, Thumann \cite{Th} uses the theory of {\em operads} in order to construct Thompson-like groups, which in this setting arise as the fundamental group of a certain category associated to a braided (resp.\ (non-)symmetric) operad (cf.\ \cite{FioreLeinster}).

The purpose of this paper is to establish a dictionary between cloning systems and certain algebraic objects called {\em action operads} \cite{Zhang,Corner-Gurski,Yoshida,Yau},  which are operads with a compatible, group-theoretical structure. Our first result is the following: 

\begin{taggedtheorem}{\ref{thm:OperadtoCS}}
Every action operad gives rise to a cloning system.
\end{taggedtheorem}

In fact, we precisely determine to which extent a converse to the above theorem holds; more concretely, we will prove that every action operad comes from a cloning system that admits a certain extra structure, and  that we call {\em restricted operadic cloning system}, see Section \ref{section:cloning_systems} and Proposition \ref{prop:CloningtoOper}. We stress that many of the known cloning systems are, in fact, operadic and bilateral; see Section \ref{section:cloning_systems} for examples and non-examples. In this language, our main results may be summarized as follows: 

\begin{taggedtheorem}{\ref{thm:BijectionActOpdsAndOperadicCS}}\label{thm:main2}
There is an explicit bijective correspondence between action operads and restricted operadic cloning systems.
\end{taggedtheorem}

As will become apparent, our methods actually yield the equivalence of the categories of restricted operadic cloning systems
and action operads, respectively.

There is a wider class of cloning systems that yield operads: the \emph{operadic} cloning systems. In fact, operads that arise from these cloning systems comply with all the roles of action operads (i.e., operads that support the equivariance of other operads). These will be defined in Section \ref{section:action_operads} with the name of \emph{general action operads}.

\medskip

\noindent {\bf Further results.} Action operads have been related to crossed simplicial groups in \cite{Zhang} and to crossed interval groups in \cite{Yoshida}. In Section \ref{section:crossed}, we review this relationship and we extend it to cloning systems. Finally, in Section \ref{section:props} we introduce a third construction using PROs ({\em product categories}) which also yields operadic cloning systems. The following diagram summarizes the relation between all of these results: 
$$
\begin{tikzcd}[ampersand replacement=\&]
   \begin{matrix}
       \text{action operads}
\end{matrix}\ar[rr,leftrightarrow,"\text{Thm }\ref{thm:BijectionActOpdsAndOperadicCS}"]\ar[d] \&\& 
   \begin{matrix}
       \text{restricted operadic}\\
       \text{cloning systems}
   \end{matrix}\ar[d] \ar[rr,leftrightarrow, "\text{Thm }\ref{thm:Bijection PROs and CSs}"] \&\& 
    \begin{matrix}
        \text{restricted} \\
       \text{cloning PROs}
   \end{matrix}
   \\
   \begin{matrix}
       \text{general}\\
       \text{action operads}
   \end{matrix}\ar[rr,leftrightarrow,"\text{Thm }\ref{thm:BijectionActOpdsAndOperadicCS}"] \ar[d]
   \&\& 
   \begin{matrix}
       \text{operadic} \\ \text{cloning systems}
   \end{matrix} \ar[d] \&\& \text{cloning PROs} \ar[ll, leftrightarrow, "\text{Thm }\ref{thm:Bijection PROs and CSs}"'] \ar[u, leftarrow] \ar[d]
   \\
     \begin{matrix}
       \text{inert crossed}\\
       \text{demi-interval groups}
   \end{matrix}\ar[rr,leftrightarrow, "\text{Prop }\ref{prop: Bilateral CS to inert crossed demi-interval gp}"]
   \&\&
   \begin{matrix}
    \text{bilateral}\\ \text{cloning systems}
   \end{matrix}\ar[rr]
   \&\& \text{cloning systems}
\end{tikzcd}
$$

\medskip

\noindent{\bf Future work.} This is the first of two papers devoted to the relation between operads and Thompson groups. In a forthcoming paper we will show that, if $A$ is an action operad and $C$ is its associated cloning system, then the fundamental group of a certain $A$\nobreakdash-operad is isomorphic to the Thompson group of the cloning system $C$.

\medskip

\noindent{\bf Plan of the paper.} In Section \ref{section:cloning_systems} we give a brief introduction to cloning systems, and give some examples. Section \ref{section:action_operads} offers an abridged overview of action operads. Section \ref{sec:braids} is devoted to the proof of Theorem \ref{thm:main2} in the special case of braid groups. These ideas are then generalized in Sections \ref{sec:operadtocloning} and \ref{sect:FromCStoActionOperads}. Finally, in Sections \ref{section:crossed} and \ref{section:props}, we will give further interpretations of our results in terms of crossed interval groups and PROs, respectively.

\medskip

\noindent{\bf Acknowledgements.} This project started with some informal conversations at the {\em IX Encuentro de J\'ovenes Top\'ologos}, held in Seville in 2021. We are grateful to the organization for their hospitality and support. The second author thanks An\'ibal Medina for a very enlightening conversation about Joyal duality. Finally, it is our pleasure to thank the referees for a great number of insightful comments and
suggestions that greatly helped improving this paper.

\section{Cloning systems}\label{section:cloning_systems}

In this section we offer a brief introduction to Witzel--Zaremsky's  cloning systems \cite{WZ}, and define their {\em bilateral} counterparts; we refer the interested reader to \cite{WZ,Zaremsky} for a detailed account on cloning systems. 

We start with a specific example, which appears as Example 2.9 in \cite{WZ}, and that will serve to establish some notation for the sequel. In what follows, $\Sigma_n$ stands for the symmetric group on $n$ elements. 

\begin{examp}[Cloning system for symmetric groups] Let $\Sigma_{\bullet}=\{\Sigma_n\}_{n\ge 1}$ be the family of symmetric groups. For every $n\ge 1$, let $\lambda_{n}:\Sigma_n \to \Sigma_{n+1}$ be the injective homomorphism obtained by fixing the last element, that is,
$$
\lambda_n(\sigma)(i)=\sigma(i) \mbox{ for $1\le i\le n$ and } \lambda_n(\sigma)(n+1)=n+1,
$$
for every $\sigma\in\Sigma_n$. For every $n\ge 1$ and $1\le j\le n$, let $c^n_j: \Sigma_n\to \Sigma_{n+1}$ be the injective map given by, thinking about permutations pictorially as strand diagrams, ``repeating'' the $j$-th strand; that is, 
$$
c_j^n(\sigma)(i)=
\left\{
\begin{array}{ll}
\sigma(i) & \mbox{ if $i\le j$ and $\sigma(i)\le\sigma(j)$,} \\
\sigma(i)+1 & \mbox{if $i<j$ and $\sigma(i)>\sigma(j)$,}\\
\sigma(i-1) & \mbox{if $i>j+1$ and $\sigma(i-1)<\sigma(j)$,}\\
\sigma(i-1)+1 & \mbox{if $i\ge j+1$ and $\sigma(i-1)\ge \sigma(j)$,}
\end{array}
\right.
$$
for every $\sigma\in\Sigma_n$. Notice that $c^n_j$ is not a group homomorphism. As shown in \cite[Example 2.9]{WZ} the families of morphisms $\lambda$ and $c$ interact with each other, and satisfy certain obvious compatibility properties, detailed in \cite[Proposition 2.6]{WZ}. The \emph{cloning system} for the family of symmetric groups is the triple $(\Sigma_\bullet, \lambda, c)$, subject to these compatibility conditions. 
\end{examp}

The maps $c^n_j$ above are called {\em cloning maps}, for obvious reasons. The notion of a cloning system is a generalization of the above example to arbitrary families of groups.

\begin{defn}\label{defn:CloningSystem}A \emph{cloning system} is a quadruple $(\G_{\bullet},\iota,\kappa,\pi)$, where
\begin{itemize}
    \item $\G_{\bullet}=\{\G_n\}_{n\geq 1}$ is a family of groups,
    \item $\iota=\{\iota_n\colon\G_n\to\G_{n+1}\}_{n\ge 1}$ is a family of  homomorphisms,
    \item $\kappa=\{\kappa^n_{j}\colon\G_{n}\to\G_{n+1}\}_{n\ge 1,\,1\leq j\leq n}$ is a family of maps, called \emph{cloning maps}, and
    \item $\pi= \{\pi_n \colon \G_n\to \Sigma_n\}_{n\ge 1}$ is a  family of homomorphism,
\end{itemize}
 subject to the following compatibility conditions:
\begin{itemize}
    \item[(i)]\label{it:i} $\pi_{n+1} \circ \iota_n=\lambda_n \circ \pi_n$, for all $n\ge 1$;  
    \item[(ii)]\label{it:ii} $(\pi_{n+1}(\kappa^n_j(g)))(i) = (c^n_j(\pi_n(g)))(i)$, for all $n\ge 1$ and all $1\le j\le n$, all $i\neq j,j+1$, all $g\in G_n$;
    \item[(iii)]\label{it:iii} $\iota_{n+1}\circ\kappa^n_j= \kappa_j^{n+1}\circ \iota_n$, for all $n\ge 1$ and all $1\le j\le n$;
    \item[(iv)]\label{it:iv} $\kappa^{n+1}_{j+1} \circ \kappa_l^n = \kappa_l^{n+1} \circ \kappa_{j}^{n}$, for all $n$ and all $l< j\le n$;
    \item[(v)]\label{it:v} $\kappa^n_j(g\cdot h)=\kappa^n_{\pi_n(h)(j)}(g)\cdot\kappa^n_j(h)$, for all $n$, all $g,h \in \G_n$ and all $j\le n$.
\end{itemize}
\end{defn}
\begin{rem} It is important to note some differences between Definition \ref{defn:CloningSystem} and the definition of cloning system of~\cite{WZ,Zaremsky}. First, we use a functional convention for composition of maps, that is, the composition of two functions $f\colon X\to Y$ and $g\colon Y\to Z$ is denoted by $g\circ f$, defined as $(g\circ f)(x)=g(f(x))$. Moreover, all our figures represent compositions from top to bottom. Second, 
the definition of cloning system in \cite{Zaremsky} requires \emph{injective} homomorphisms $\iota_{n,m}\colon \G_{n}\to \G_{m}$ for every $m> n\geq 1$. First, the equivalence between Definition \ref{defn:CloningSystem} and this one implies that it suffices to consider injective maps $\iota_{n,n+1}=\iota_n$ for all $n\geq 1$. And second, the injectivity condition is not essential, and hence we may remove this assumption altogether (cf.\ Remark \ref{rem:crossed_injectivity}).  
\end{rem}

Witzel and Zaremsky observe that cloning systems often satisfy the following strengthed version of these axioms. 
\begin{itemize}
   	\item[(iv+)] $\kappa^{n+1}_{j+1}\circ \kappa^n_{j} = \kappa_{j}^{n+1}\circ \kappa_j^n$, for all $1\le j\le n$;
    \item[(vii)]\label{it:iii'} $\iota_{n+1}\circ \iota_n = \kappa_{n+1}^{n+1}\circ \iota_n$, for all $n\ge 1$;    
\end{itemize}
\begin{rem}
    Condition (vii) was part of the original definition of cloning system: In~\cite[Definition 2.18]{WZ} the cloning maps $\kappa$ are required to be a \emph{family of cloning maps}, which must satisfy two conditions presented at the beginning of page 309 in that paper. These two conditions correspond to Conditions (iii) and (vii) in this article.

In \cite[page 308]{WZ} (see also \cite{Bri07}) the hedge monoid ${\mathcal H}$ is introduced, together with a surjective map ${\mathcal F}\to {\mathcal H}$ from the monoid of forests to the monoid of hedges. The colimit of the groups in a cloning system comes with an action of ${\mathcal F}$. Condition~(iv+) in this article is equivalent to requiring that that action factors through the hedge monoid (which holds for most examples; see Observation 2.11 and paragraph before Observation 2.19 in \cite{WZ}).

\end{rem}

We now introduce the notion of a \emph{bilateral} cloning system. In a nutshell, in the same way that the maps $\iota$ of Definition \ref{defn:CloningSystem} informally correspond to ``adding elements on the right'', a bilateral cloning system comes equipped with  a family of ``dual'' maps $\zeta$ that correspond to ``adding elements on the left''. 

We now proceed to formalize this idea. For the family of symmetric groups, we denote by $\rho_n: \Sigma_n \to \Sigma_{n+1}$ the injective homomorphism that fixes the first element, that is, 
$$
\rho_n(\sigma)(i)=\sigma(i-1)+1 \mbox{ for $2\le i\le n+1$ and } \rho_n(\sigma)(1)=1,
$$
for every $\sigma\in\Sigma_n$.

\begin{defn}\label{defn:NCCloningSystem} A \emph{bilateral cloning system} is a quintuple $(\G_{\bullet},\iota, \zeta, \kappa,\pi)$, where $(\G_{\bullet},\iota,\kappa, \pi)$ is a cloning system satisfying conditions (iv+) and (vii), and $\zeta=\{\zeta_n\colon \G_n\to\G_{1+n}\}_{n\ge 1}$ is an additional family of homomomorphisms satisfying the following conditions (see Figures \ref{fig:AxiomVI} and \ref{fig:AxiomVIII}):
\begin{itemize}
    \item[(i')] $\pi_{n+1} \circ \zeta_n=\rho_{n}\circ \pi_n $, for all $n\ge 1$;  
    \item[(iii')] \label{it:iiib} $\zeta_n\circ \kappa_j^n = \kappa_{j+1}^{n+1}\circ\zeta_{n}$, for all $n\ge 1$ and all $1\le j\le n$;
    \item[(vi)] $\zeta_{n+1}\circ \iota_n = \iota_{n+1}\circ\zeta_{n}$, for all $n\ge 1$;
    \item[(vii')]\label{it:iiib'} $\zeta_{n+1}\circ \zeta_n = \kappa_{1}^{n+1}\circ \zeta_n$, for all $n\ge 1$;
 \end{itemize}
 A bilateral cloning system is called \emph{restricted} if it additionally satisfies the following condition:
 \begin{itemize}
        \item[(ii+)] $\pi_{n+1} \circ \kappa^n_j = c^n_j \circ \pi_n$, for all $n\ge 1$ and all $1\le j\le n$;
 \end{itemize}
\end{defn}

For our last notion in this section, we introduce the following maps, whose motivation will become clear in Remark \ref{rem:properties}.
\begin{notat}\label{notat: generalized structure maps of a bilateral cloning system}
 Let $(\G_{\bullet},\iota, \zeta, \kappa,\pi)$ be a bilateral cloning system. We will define some maps and exemplify them in the case of the cloning system of symmetric groups. For that, recall that each partition of $\{1,\ldots,n+r\}$ into $n$ blocks yields a homomorphism $\Sigma_n\to \Sigma_{n+r}$ that sends a permutation of $n$ elements to the permutation of the $n$ blocks. For example, the partition $\fbox{1}\fbox{23}\fbox{45}$ of $\{1,2,3,4,5\}$ yields a map $\Sigma_{3}\to \Sigma_5$, and the image of the permutation $(2,1,3)$ is the permutation $(2,3,1,4,5)$.

\begin{itemize}
    \item $\iota_n(r)\colon \G_n\to \G_{n+r}$  is defined as the  composition $\iota_n(r)=\iota_{n+r-1}\circ\cdots\circ\iota_{n}$, for all $n,r\ge 1$. If $\G=\Sigma$, then $\iota_n(r)(g)$ is obtained from $\iota_n(g)$ and the block decomposition 
    \[
    \fbox{$\!\!\phantom{(}1\!\!\phantom{)}$}\ldots\fbox{$\!\!\phantom{(}n\!\!\phantom{)}$}\fbox{$(n+1)\ldots (n+r)$}.
    \]
     \item $\zeta_n(l)\colon\G_{n}\longrightarrow\G_{l+n}$ is defined as the composition $
\zeta_n(l)=\zeta_{n+l-1}\circ\cdots\circ\zeta_{n}$, for all $n,l\ge 1$. If $\G=\Sigma$, then $\zeta_n(l)(g)$ is obtained from $\zeta_n(g)$ and the block decomposition
    \[
        \fbox{$\!\!\phantom{(}1\ldots  l\!\!\phantom)$}\fbox{$(l+1)$}\ldots\fbox{$(n+l)$}.
    \] 
    \item $\kappa_j^n(m)\colon \G_n\to \G_{n+m-1}$ is defined as  $\kappa_j^n(m)=\kappa_j^{n+m-2}\circ\cdots\circ \kappa_j^{n}$, for all $n,m\ge 1$ and all $1\le j\le n$. If $\G=\Sigma$, then $\kappa_j^n(m)(g)$ is obtained from $g$ and the block decomposition 
    \[
        \fbox{\!\!\phantom{(}1\!\!\phantom{)}}\ldots \fbox{$(j-1)$}\fbox{$j\ldots (j+m-1)$}\fbox{$(j+m)$}\ldots\fbox{$(n+m-1)$}.
    \]
    \item $\nu^n_j(m)\colon\G_{n}\to\G_{n+m-1}
$ is defined as $\nu_j^n(m)=\zeta_{m+n-j}(j-1)\circ\iota_n(m-j)=\iota_{n+j-1}(m-j)\circ\zeta_n(j-1)$,  for all $m,n\ge 1$ and $1\leq j\leq m$. If $\G=\Sigma$, then $\nu^n_j(m)$ is obtained from $\zeta_{n+1}\circ\iota_n(g)$ and the block decomposition
\[
\fbox{$1\ldots (j-1)$}\fbox{$\!\!\phantom{(}j\!\!\phantom{)}$}\ldots\fbox{$(j+n-1)$}\fbox{$(j+n)\ldots (n+m-1)$}.
\]
\end{itemize}
    
Note that the morphisms $\zeta_n(l)$ and $\nu_j^n(m)$ only make sense for bilateral cloning systems and that, by convention, we consider $\iota_n(0)=\zeta_n(0)=\kappa_j^n(1)=\id_{\G_n}$.
\end{notat}

\begin{defn}\label{defn:OperadicCloningSystem} An 
\emph{operadic cloning system} is a bilateral cloning system for which the following conditions are satisfied (see Figures \ref{fig:AxiomX} and \ref{fig:AxiomXI}):
\begin{itemize}
	\item[(viii)]$\kappa_i^n(m)(g)\cdot \nu_i^m(n)(h) = \nu^m_{\pi(g)(i)}(n)(h)\cdot \kappa^n_i(m)(g)$, for all $m,n\geq 1$ and all $g\in \G_n$ and $h\in\G_m$;
	\item[(ix)] $\iota_n(m)(g)\cdot \zeta_m(n)(h)=\zeta_m(n)(h)\cdot \iota_n(m)(g)$, for all $m,n\geq 1$ and all $g\in\G_n$ and $h\in\G_m$.
\end{itemize}
\end{defn}


In order to get a grip on the intuition behind the definitions above, we next describe perhaps the primordial example of a (bilateral) cloning system, namely that of braid groups; we refer the reader to \cite{WZ} for details.   

\begin{examp}[Braid groups]\label{examp:cloning}
    
 Let $\Br_n$ denote the braid group on $n$ strands. For every $n\ge 1$ there is a canonical surjective group homomorphism $\pi_n\colon\Br_n\to\Sigma_n$ by sending each braid to its underlying permutation. We also have inclusion maps $\iota_n\colon \Br_n\to\Br_{n+1}$ corresponding to ``adding one strand on the right'', and cloning maps
$$
\kappa_{j}^n\colon \Br_n\longrightarrow \Br_{n+1}
$$
given by duplicating the $j$-th strand into two parallel strands (commonly known as \emph{cabling} in the literature); see Figure \ref{fig:Kappa Braid} for an example, and \cite[Example~3.3]{Zaremsky} for details. Recall that we always draw braids and compositions from top to bottom.
\begin{figure}[htp]
    \centering
    \includegraphics[width=6cm]{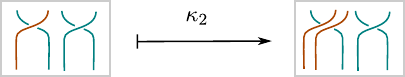}
    \caption{Cloning map $\kappa_2^4$ on $\Br_4$.}
    \label{fig:Kappa Braid}
    \end{figure}

Together, the three families of maps defined above endow the collection of all braid groups $\Br_{\bullet}=\{\Br_n\}_{n\ge 1}$ with the cloning system structure $\Br=(\Br_{\bullet},\iota,\kappa,\pi)$, see \cite{WZ} for a proof. 

Moreover, in analogy with the maps $\iota_n$, we can also define inclusion maps $\zeta_n\colon \Br_n\to\Br_{n+1}$ that informally correspond to ``adding one strand on the left''. Equipped with these maps, one readily checks that $\Br=(\Br_{\bullet},\iota,\zeta,\kappa,\pi)$ is a restricted bilateral cloning system. Moreover, it is an operadic cloning system. For illustrative purposes, Figures \ref{fig:AxiomVI} to \ref{fig:AxiomXI} depict particular instances of some of the conditions of the bilateral cloning system structure of $\Br_{\bullet}$.
\end{examp}  
 \begin{figure}
 \centering
 \includegraphics[width=7cm]{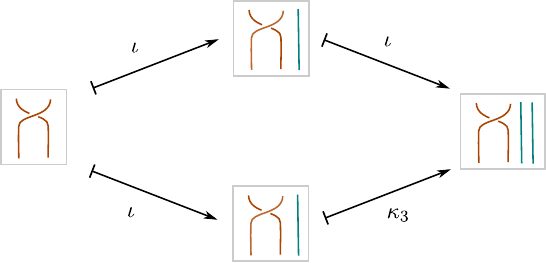}
    \caption{Condition (vii), $\kappa_3^3\circ \iota=\iota\circ\iota$.}
    \label{fig:AxiomVI}
    \end{figure}
 \begin{figure}
    \centering
    \includegraphics[width=7cm]{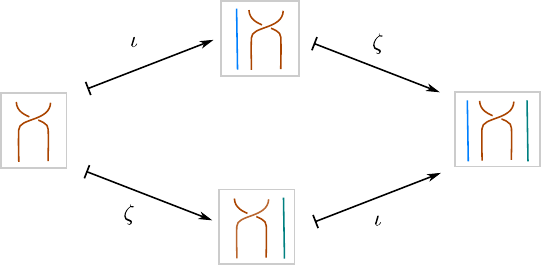}
    \caption{Condition (vi), $\iota\circ \zeta=\zeta\circ \iota$.}
    \label{fig:AxiomVIII}
    \end{figure}
\begin{figure}
    \centering
    \includegraphics[width=10cm]{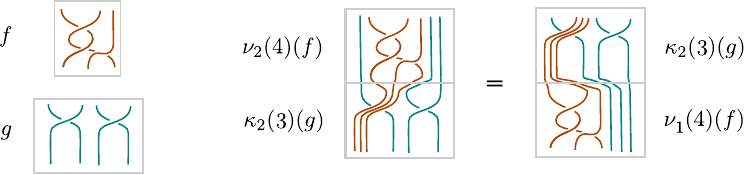}
    \caption{Condition (viii).}
    \label{fig:AxiomX}
    \end{figure}
 \begin{figure}
    \centering
    \includegraphics[width=10cm]{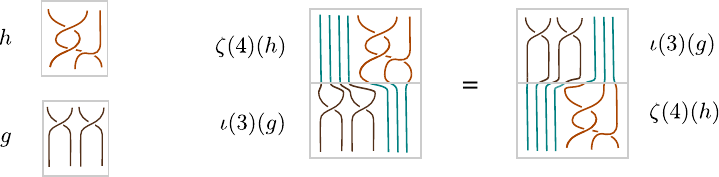}
    \caption{Condition (ix).}
    \label{fig:AxiomXI}
    \end{figure}

\begin{examp}\label{examp:CloningSystem of Signed Permutations}(Signed symmetric groups)
     Let $\Sigma^{\pm}_n$ denote the signed symmetric group, that is, the set of permutations $g$ of the set $\{1,-1,2,-2,\ldots,n,-n\}$ such that $g(i) = - g(-i)$; see \cite{Zaremsky}. First of all, the morphism $\pi_n\colon \Sigma^{\pm}_n\to \Sigma_n$ is determined by $\pi_n(g)(i)=\vert g(i)\vert$. The structure map $\zeta_n$ is determined by $\zeta_n(g)(n+1)=n+1$, and $\zeta_n(g)(i)=g(i)$ when $1\leq i\leq n$. The maps $\iota_n$ are defined analogously, by inserting a fixed first element instead of a last one. Finally, the maps $\kappa_j^n$ are defined as follows:
     
     \[
        \kappa_j^n(g)(i) = \begin{cases}
            g(i) & \text{if $1\leq i<j$}, \\
            g(i) & \text{if $i=j$ and $g(i)>0$}, \\
            g(i)+1 & \text{if $i=j+1$ and $g(i)>0$}, \\
            g(i)-1 & \text{if $i=j$ and $g(i)<0$}, \\
            g(i) & \text{if $i=j+1$ and $g(i)<0$}, \\
            g(i-1) & \text{if $j+1<i\leq n+1$}.
        \end{cases}
     \]
 Equipped with these maps, $\Sigma^{\pm}=(\Sigma^{\pm}_{\bullet},\iota,\zeta,\kappa,\pi)$ is a bilateral cloning system, which is \emph{not} restricted. Moreover, it is an operadic cloning system.
\end{examp}

Other examples of families of groups admitting a bilateral cloning system structure are the {\em mock symmetric groups} and the {\em loop braid groups} (also known as {\em symmetric automorphisms of free groups}); see \cite{WZ}; as well as the {\em twisted braid groups} \cite{Zaremsky}. 
The latter two bilateral cloning systems are not restricted.
\medskip

Next, we discuss two examples of (bilateral) cloning systems from \cite{WZ} and \cite{Zaremsky} which, as we will see, are not operadic.

\begin{examp}[Direct powers]\label{examp:DirectPowers}
Let $G$ be a group and denote by $G^n$ the $n$-fold direct product of $G$ with itself. Write  $\iota_n: G^n \to G^{n+1}$ for the map that adds the identity (in $G$) as last entry, let $\pi_n: G^n \to \Sigma_n$ the trivial homomorphism, and consider the map $\kappa^n_j:G^n \to G^{n+1}$ that duplicates the $j$-th entry. Then, the quadruple $(\{G^n\}_{n\ge 1}, \iota, \pi, \kappa)$ is  a cloning system on the set of direct powers of $G$. 

Despite the fact that there is an obvious map $\zeta_n: G^n \to G^{n+1}$ that adds the identity (in $G$) as first entry, one may check that the maps $\kappa$ and $\nu$ satisfy conditions (viii) and (ix) of the definition of an operadic cloning system iff $G$ is abelian; e.g.\ take $n=m=1$ in those axioms to obtain the direct implication.
\end{examp}

\begin{examp}[Upper triangular matrices]\label{examp:UpperTriangularMatrices}
Let $U_n$ denote the group of invertible $n\times n$ upper triangular matrices with real coefficients. Consider the obvious inclusion map $\iota_n: U_n \to U_{n+1}$ given by adding a 1 as the lowermost element on the diagonal. 

There are cloning maps $\kappa_j^n: U_n \to U_{n+1}$ that informally correspond to a certain duplication of the $j$-th column that preserves the upper triangular structure of the matrix, and which becomes apparent just by giving the following particular example; see \cite{Zaremsky} for details: 

\[
\kappa^3_2\begin{pmatrix}
1 & 2 & 3 \\ 0 & 4 & 5 \\ 0 & 0 & 6
\end{pmatrix}  = 
\begin{pmatrix}
1 & 2 & 2 & 3 \\ 0 & 4 & 0 & 0 \\ 0 & 0 & 4 & 5 \\ 0 & 0 & 0 & 6
\end{pmatrix}
\]
Setting $\pi_n: U_n \to \Sigma_n$ to be the trivial homomorphism, the set $U_\bullet = \{U_n\}_{n \ge 1}$ acquires the cloning system structure $U = (U_\bullet, \iota, \kappa, \pi)$. 

Observe that one could define another inclusion map $\zeta_n: U_n \to U_{n+1}$ by adding a 1 as the uppermost element of the diagonal; however, as in the previous example, the interaction of the maps $\zeta$ and $\kappa$ does not satisfy condition (viii) of the definition of an operadic cloning system.
\end{examp}

\section{Action operads}\label{section:action_operads}
We start this section recalling the classical notion of an {\em operad}:

\begin{defn}\label{defn:Operad} 
A \emph{symmetric operad} $\EE$ on sets is a triple $(\EE, \{\circ_i\}_i, \id)$, where
\begin{itemize}
\item $\EE=\{\EE(n)\}_{n\ge 1}$ is a family of sets, and each $\EE(n)$ is equipped with a right $\Sigma_n$-action, for every $n\ge 1$,
\item ${\id}\in \EE(1)$ is called the \emph{unit} of the operad, and \item $\{\circ_i\}_i$ is a family of maps
$$
\circ_i\colon \EE(n)\times \EE(m)\to \EE(n+m-1) \quad\mbox{for all $n,m\ge 1$ and $1\le i\le n$},
$$ called \emph{$\circ_i$-operations} or \emph{partial composition products},
such that ${\id} \circ_i y=y$ and $x\circ_i {\id}=x$ for every $x,y\in \EE(n)$. Moreover, these $\circ_i$\nobreakdash-operations satisfy certain associativity and equivariance axioms, which are spelled out in, for example, \cite[Definition 11]{Markl}.
\end{itemize}
If we forget about all the symmetric group actions on the sets $\EE(n)$, we have the notion of a \emph{non-symmetric operad}.

\medskip

A \emph{morphism of operads} $f\colon\EE\to \mathsf{P}$ consists of maps $f_n\colon \EE(n)\to \mathsf{P}(n)$ for $n\ge 1$, that are compatible with the unit and $\circ_i$-operations of $\EE$ and $\mathsf{P}$, plus the $\Sigma_n$-actions in the case of symmetric operads. 
\end{defn}

\begin{rem} There is an alternative definition of (non-)symmetric operad, see \cite[Definition 1]{Markl}, that replaces the maps $\circ_i\colon \EE(n)\times \EE(m)\to \EE(n+m-1)$ by full-composition products
$$
\EE(n)\times\big(\EE(m_1)\times \dots\times \EE(m_n)\big)\longrightarrow \EE(m_1+\dots +m_n).
$$
As shown in \cite[Proposition 13]{Markl}, both definitions are interchangeable if one considers units, as we do. We prefer to use the version with partial composition products because it is better suited for the results of this paper.
\end{rem}

\begin{rem} Note that Definition \ref{defn:Operad} avoids nullary operations in an operad, i.e.\ there is no $\EE(0)$ in $\EE$. In this work, we will only consider operads without nullary operations, or in other words, without constants. This choice is not essential, but it is made to have a clearer connection between action operads and cloning systems.
\end{rem}

\begin{examp}[Symmetric groups]
The family of symmetric groups $\Sigma_{\bullet}=\{\Sigma_n\}_{n\ge 1}$ has the structure of a non-symmetric operad on sets, with $\id\in\Sigma_1$ the trivial permutation of $\Sigma_1$. The partial composition products $\circ_i\colon\Sigma_n\times \Sigma_m\to \Sigma_{m+n-1}$ are defined as follows: if $\sigma\in \Sigma_n$ and $\tau\in\Sigma_m$, then $\sigma\circ_i\tau$ is the permutation of $\Sigma_{m+n-1}$ obtained by ``inserting'' $\tau$ in $\sigma$ at the $i$-th place as a block and rearranging the indices accordingly. Figure~\ref{fig:CompositionProduct Sigma} shows an example of the composition product $\circ_2\colon\Sigma_4\times \Sigma_3\to\Sigma_6$.
\begin{figure}[h]
    \centering
    \includegraphics[width=5cm]{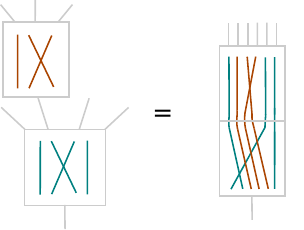}
    \caption{\scriptsize
    $(1,3,2,4)\circ_2(1,3,2)=(2,3,5,4,1,6)$.
    }
    \label{fig:CompositionProduct Sigma}
    \end{figure}
\end{examp}

\subsection{Action operads}
We now introduce the concept of an {\em action operad}, which essentially is a family of groups satisfying certain properties that allow one to define operads with equivariance relative to this family; the reader is referred to \cite{Corner-Gurski} for a detailed treatment of action operads and their properties. They have been also studied under the name \emph{group operads} \cite{Zhang, Yoshida, Yau}.

\begin{defn}\label{defn:ActionOperad} An \emph{action operad without constants} $\G$, or simply an \emph{action operad}, is a quadruple $(\G_{\bullet},\pi,\{\circ_i\}_i,\id)$, where
\begin{enumerate}
    \item\label{cond:ac:1} $\G_\bullet=\{\G_n\}_{n\ge 1}$ is a family of groups, and $(\G_{\bullet},\{\circ_i\}_i,\id)$ is a non-symmetric operad on sets, where $\{\circ_i\}_i$ are the partial composition products of the operad and $\id\in \G_1$ is the unit. The associativity and unitality of the partial composition products assert that if $f\in \G_n$, $g\in \G_m$ and $h\in \G_l$, then 
\begin{align*}
(f\circ_i g)\circ_j h &= \begin{cases}
(f\circ_j h) \circ_{i+l-1}g & \text{if $j<i$},\\
 (f \circ_{j-m+1} h)\circ_i g& \text{if $j\geq i+m$},\\
f\circ_i (g\circ_{j-i+1}h) & \text{if $j=i,\ldots, i+m-1$},
\end{cases} \\
f\circ_i \id &= f, \\
\id\circ_1 g &= g;
\end{align*}
    \item\label{cond:ac:2} $\pi: \G_\bullet \to \Sigma_{\bullet}$ is a map of operads which is also a levelwise group homomorphism, that is, $\pi_n\colon\G_n \to \Sigma_n$ is a homomorphism for all $n\ge 1$;
\item\label{cond:ac:3} For every $f, f' \in \G_n$, $g,g'\in \G_{m}$, we have that
\begin{align*}
(f\cdot f')\circ_i(g\cdot g') &= (f\circ_{\pi_n(f')(i)}g)\cdot (f'\circ_i g'),
\end{align*}
with the multiplication taking place in the group $\G_{n+m-1}$.
\end{enumerate}
\label{def:actionoperad}
\end{defn}

 Note that the partial composition products $\circ_i$ are not group homomorphisms in general, that action operads are not assumed to be symmetric operads, and that they have no nullary operations. It follows from the axioms that the unit element $\id\in \G_1$ of the operad $\G_{\bullet}$ is precisely the unit $e_1$ of the group $\G_1$. Analogously, we will denote by $e_n$ the unit of the group $\G_n$.

\begin{notat}\label{notat: Dropping pi when understood} Observe that, both in a cloning system and in an action operad, the group $\G_n$ acts on the set $\{1,\ldots,n\}$ via the homomorphism $\pi_n$ for every $n$. If these maps $\pi_n$ are understood from the context, for $g\in \G_n$ we will write $g(i)$ instead of $\pi_n(g)(i)$. 
\end{notat}

\begin{rem}\label{rem: pi trivial for action operad}
    Finally, note that an action operad with trivial $\pi$ is the same thing as a non-symmetric operad on  groups.
\end{rem}

\subsection{General action operads} Inspired by the definition of cloning system, we relax as follows the definition of action operad. Note that this relaxation does not diminish the ability of such an operad to act on other operads (see the beginning of Section \ref{subsect:Fundamental gps of G-operads}).  
In fact, we believe that this should be the adequate definition of action operad. 
\begin{defn}\label{defn:general_action_operad}
    A \emph{general action operad} $\G$ is a tuple $(\G_\bullet,\pi,\{\circ_i\}_i,\id)$ satisfying Conditions \eqref{cond:ac:1} and \eqref{cond:ac:3} together with the following:
    \begin{enumerate}
 \item[(2')]
 $\pi_n\colon \G_n\to \Sigma_n$ is a group homomorphism for each $n\geq 1$ such that if $g\in \G_n$ and $h\in \G_m$, then the action of the symmetric group elements $\pi_n(g)\circ_i \pi_m(h)$ and $\pi_{n+m-1}(g\circ_i h)$ coincide on $$
        \{1,2,\ldots,n+m-1\}\backslash \{i,\ldots, i+m-1\}.
        $$
    \end{enumerate}
    A \emph{morphism of (general) action operads} is a map of operads $\phi\colon \G\to \widehat{\G}$ satisfying:
\begin{itemize}
    \item[{\rm(i)}] $\phi_n\colon \G_n\to \widehat{\G}_n$ is a group homomorphism for any $n\geq 1$, and 
    \item[{\rm(ii)}] $\widehat{\pi}_n \phi_n=\pi_n$ for any $n\geq 1$. 
    \end{itemize}
\end{defn}
Here is an example of a general action operad that is \emph{not} an action operad. It corresponds to the cloning system of signed symmetric groups or to the hyperoctahedral inert demi-interval group (see Section \ref{section:crossed}).

\begin{examp}\label{examp:Signed Permutations}
    The signed symmetric groups $\Sigma^{\pm}_n$ from Example \ref{examp:CloningSystem of Signed Permutations} can be equipped with a general action operad structure. The composition $g\circ_i h$ is defined as follows: if $g(i)$ is positive for $i\geq 1$, then $g\circ_i h$ is obtained by cabling the signed symmetry $h$ in the $i$-th strand of the permutation $g$. If $g(i)$ is negative for $i\geq 1$, then $g\circ_i h$ is obtained by cabling the signed symmetry $h'$ in the $i$-th strand of $g$, where $h'(j) = -h(m-j+1)$ for a positive $j$. A signed permutation $g$ is determined by the pair $(\sigma,A)$, where $\sigma$ is the underlying permutation and $A$ is the set of positive $i$'s such that $g(i)$ is negative. With this notation, the composition of $(g,A)$ and $(g',A')$ is $(g\circ g',g'(A)\triangle A')$, where the symbol $\triangle$ denotes symmetric difference.
\end{examp}

From the classification of inert crossed interval groups and the relation between action operads and crossed interval groups, it is reasonable to think that the signed symmetric general action operad is final in the category of general action operads. Motivated by this observation, we provide a simplification of the definition of general action operad which, as a byproduct, shows that this is indeed the case.

Let us note that given a general action operad $\G=(\G_\bullet,\pi,\{\circ_i\}_i,\id)$, we can define a family of maps $\pi^{\pm}\colon \G_{\bullet}\to \Sigma^{\pm}_{\bullet}$ as follows: For an element $g\in \G_n$, define $\pi^{\pm}_n(g) = (\pi_n(g),A)$ with $A = \{j\mid \pi_{n+1}(g\circ_j e_2)\neq \pi_n(g)\circ_j e_2\}$.

\begin{lem}\label{lem: Simplification of general action operad axioms} Condition {\rm (2')} in the definition of general action operad $\G$ can be replaced by the following: 
    \begin{itemize}
        \item[($2^{\pm}$)] $\pi^{\pm}: \G_\bullet \to \Sigma^{\pm}_{\bullet}$ is a map of operads and a levelwise group homomorphism. 
    \end{itemize}  
    Moreover, Condition {\rm (ii)} in the definition of a morphism of general action operads $\phi\colon \G\to \widehat{\G}$ is equivalent to requiring:
     \begin{itemize}
        \item[$(\mathrm{ii}^{\pm})$] $\widehat{\pi}^{\pm}_n\phi_n=\pi_n^{\pm}$ for any $n\geq 1 $. 
    \end{itemize}
\end{lem}
\begin{proof} Suppose that $\pi\colon \G_{\bullet}\to \Sigma_{\bullet}$ is a family of maps satisfying Condition (2'). We check that the associated family of maps $\pi ^{\pm}\colon \G_{\bullet}\to \Sigma^{\pm}_{\bullet}$, defined right before Lemma \ref{lem: Simplification of general action operad axioms}, satisfies Condition ($2^{\pm}$).

 To check that $\pi^{\pm}_n$ is a group homomorphism, we have that, if $g,h\in \G_n$ have images $\pi^{\pm}_n(g) = (\sigma,A)$, $\pi^{\pm}_n(h) = (\sigma',A')$ and $\pi^{\pm}_n(g\cdot h) = (\sigma'',A'')$, then $\sigma'' = \sigma\cdot \sigma'$ and
    \[
        (g\cdot h)\circ_j e_2 = (g\cdot h)\circ_j (e_2\cdot e_2) = (g\circ_{\sigma'(j)} e_2)\cdot (h\circ_j e_2)
    \]
    The product $\pi_{n+1}((g\cdot h)\circ_j e_2)$ is either equal to $\pi_n(g\cdot h)\circ_j e_2$ or differs from it by a twist of the entries $\{j,j+1\}$. Inspection shows that $A'' = g'(A)\triangle A'$.

    One still has to show that $\pi^{\pm}$ is an operad map. By the multiplication rule it is enough to check it for the compositions $g\circ_j e_m$. Since the latter are iterated compositions of the form $g\circ_j e_2$, it is enough to check it in this last case, in which is true by definition.

    Finally, if $\pi^{\pm}\colon \G_{\bullet}\to \Sigma_{\bullet}^\pm$ is a family of maps satisfying Condition ($2^{\pm}$), then composing them with the homomorphisms $\Sigma_{n}^{\pm}\to \Sigma_n$ that forget the signs yields a family of maps satisfying $(2')$.

    In order to prove $(\mathrm{ii}^{\pm})$, let $g\in \G_n$ and let $j\in \{1,\ldots n\}$. We have to prove that
    \[
        \widehat{\pi}^{\pm}_n(\phi_n(g))(j) = \pi^{\pm}(g)(j)
    \]
    Since $\phi$ is a map of general action operads, we know that both sides are equal except possibly for a sign. The sign of the right hand-side is positive if $$\pi_n(g\circ_j e_2) = \pi_n(g)\circ_j e_2$$ and negative otherwise. The sign of the left hand-side is positive if $$\widehat{\pi}_n(\phi_n(g)\circ_j e_2) = \widehat{\pi}_n(\phi_n(g))\circ_j e_2.$$ and negative otherwise. Since $\phi_n$ is a morphism of general action operads, both left hand-sides agree and both right hand-sides agree, thus $(\mathrm{ii}^{\pm})$ holds.
\end{proof}

\begin{cor} The final object in the category of general action operads is $\Sigma^{\pm}_{\bullet}$. 
\end{cor}

\begin{rem}\label{rem: Condition ii' for cloning systems}
    This lemma has its counterpart in the world of cloning systems: Condition (ii) in the definition of cloning system can be replaced by the following:
    \begin{itemize}
        \item[(ii')] $\pi_n\colon \G_n\to \Sigma_{n}^\pm$ is a levelwise group homomorphism such that the identity $\pi_{n+1}\circ \kappa_j^n = c_j^n\circ \pi_n$ holds for all $n\geq 1$ and all $1\leq j\leq n$,
    \end{itemize}
where $c_j^n$ are the cloning maps of the signed symmetric cloning system. In order to make sense of axioms (v) and (viii), we are implicitly using the fact that $\Sigma^\pm_n$ acts on $\{1,2,\ldots,n\}$ through the homomorphism $\Sigma_n^{\pm}\to \Sigma_n$ that forgets the signs.
\end{rem}

\subsection{Fundamental groups of $\G$-operads and  unoriented ribbon braids}\label{subsect:Fundamental gps of G-operads}

Fix a (general) action operad $\G$. Then, one may define a $\G$-operad as in \cite[Def.~1.14]{Corner-Gurski}, \cite[Def.~4.2.6, Prop. 4.3.1]{Yau} or \cite[Def.~2.30]{Zhang}, i.e., as a non-symmetric operad $\EE$ with an action $\triangleright\colon\EE(n)\times \G_n\to \EE(n)$ such that the composition
\[
    \EE(n)\times \big(\EE(m_1)\times \ldots\times \EE(m_n)\big)\lra \EE(m_1+\ldots+m_n)
\]
is equivariant with respect to the map
\[
    \G_n\times (\G_{m_1}\times \ldots\times \G_{m_n})\lra \G_{m_1+\ldots+m_n}.
\]
Alternatively, if one considers operads with identities as we do, one requires the composition $\circ_i\colon\EE(n)\times \EE(m)\to \EE(n+m-1)$ to be equivariant with respect to the map $\circ_i\colon\G_n\times \G_m\to \G_{n+m-1}$, i.e.\ to satisfy
$$
(u\triangleright g)\circ_i (u'\triangleright g')=(u\circ_{g(i)}u')\triangleright(g\circ_ig'),
$$
for any $u\in \EE(n)$, $u'\in \EE(m)$ and $g\in \G_n$, $g'\in \G_{m}$; notice the use of Notation \ref{notat: Dropping pi when understood}.
For instance, $\G$-operads for $\G=\Sigma$ are just the symmetric operads of Definition \ref{defn:Operad}.

Our goal in this subsection is to provide more examples of (general) action operads. For that purpose, we briefly discuss a construction based on fundamental groups of topological $\G$-operads that produces (general) action operads over $\G$.


First, observe that the forgetful functor from $\G$-operads to non-symmetric operads admits a left adjoint which simply adds a free  right $\G$-action,
 $$
(-)_{\G}\colon \begin{Bmatrix}
\text{non-symmetric}\\
\text{operads}
\end{Bmatrix}  \longrightarrow \begin{Bmatrix}
\G\text{-operads}
\end{Bmatrix}, \quad \mathsf{P}\mapsto \mathsf{P}_{\G},
 $$
 where $\mathsf{P}_{\G}(n)=\mathsf{P}(n)\times \G_{n}$ and the operation $\circ_i$ is defined as the composition
 $$
 \begin{tikzcd}[ampersand replacement=\&]
     \mathsf{P}_{\G}(n)\times \mathsf{P}_{\G}(m)\ar[d,"\cong"',"\text{switch}"] \\ 
     \mathsf{P}(n)\times \mathsf{P}(m)\times \G_{n}\times\G_{m}
     \ar[d,"\circ_i\times \circ_i"] \\ \mathsf{P}(n+m-1)\times \G_{n+m-1}\ar[d, equal]\\
     \mathsf{P}_{\G}(n+m-1)
 \end{tikzcd}\quad .
 $$
Of course, the same discussion applies to topological operads.
 
 Applying this left adjoint functor to the (non-symmetric) associative operad $\mathsf{As}$, characterized by $\mathsf{As}(n)=*$ for any $n\geq 1$ (and $\mathsf{As}(0)=\emptyset$), we obtain $\mathsf{As}_{\G}$. A map from $\mathsf{As}_{\G}$ selects basepoints in a coherent way in a $\G$-operad. In fact, 
 \begin{defn} Let $\EE$ be a topological $\G$-operad. Then, a \emph{good $\G$-basepoint} for $\EE$ is a map of topological $\G$-operads $\eta\colon\mathsf{As}^h_{\G}\to \EE$, where $\mathsf{As}^h_{\G}$ is a topological $\G$-operad equipped with a operad map $\mathsf{As}^h_{\G}\to \mathsf{As}_{\G}$ which is a levelwise  homotopy equivalence.
 \end{defn}

For any $k\geq 1$ one can make a choice of basepoint $\mu_k\in \mathsf{As}_{\G}^{h}(k)$ representing the plain $k$-ary multiplication, that is, lying in the connected component corresponding to $(*,e_k)\in \mathsf{As}_{\G}(k)$. By definition, there are paths $ \mu_{n+m-1}\simeq \mu_n\circ_i\mu_m$ in the space $ \mathsf{As}_{\G}^h(n+m-1)$ and homotopies relating natural concatenations of those paths. For that reason, we fix basepoints $\epsilon_k:=\eta(\mu_k)\in \EE(k)$ for paths and loops as \cite[Section 3]{Zhang} in the sequel.


Now, assume that the action of $\G_n$ on $\EE(n)$ is a covering action (see \cite[Section 1.3]{Hatcher}) and that $\EE(n)$ is path-connected. Then, we have a homotopy fiber sequence $\G_n\hookrightarrow \EE(n)\twoheadrightarrow \EE(n)/\G_n$, which induces a short exact sequence of groups
\begin{equation}\label{eqt:ses pi1 of Goperads}
    1\xrightarrow{\;\quad} \pi_1\big(\EE(n),\epsilon_n\big)\xrightarrow{\;\quad} \pi_1\big(\EE(n)/\G_n,[\epsilon_n]\big)\xrightarrow{\;\;\partial_n\;\;} \G_n \xrightarrow{\;\quad} 1,
\end{equation}
and an isomorphism of relative homotopy groups
\begin{equation}\label{eqt:pi1 of Goperads}
    \pi_1\big(\EE(n)/\G_n,[\epsilon_n]\big)\cong \pi_1\big(\EE(n);\epsilon_n,\epsilon_n\triangleright\G_n\big).
\end{equation}
The left-hand side in (\ref{eqt:pi1 of Goperads}) has a group structure given by concatenation of loops, while the right-hand side is easily seen to form a non-symmetric operad. Altogether, we have the following generalization of \cite[Theorem 3.4]{Zhang}:
 \begin{prop} Let $(\G,\pi,\circ_i,\id)$ be a (general) action operad, $\EE$ be a topological $\G$-operad without constants so that: {\rm (i)} $\EE$ is equipped with a good $\G$-basepoint, {\rm (ii)} $\EE(n)$ is path-connected for any $n\geq 1$, and {\rm (iii)} $\G$ acts on $\EE$  via covering actions. Then, the sequence of groups $\pi_1\big(\EE(n)/\G_n,[\epsilon_n]\big)$ together with the maps 
 $$
 \pi_1(\EE(n)/\G_n,[\epsilon_n])\xrightarrow{\;\;\partial_n\;\;} \G_n \xrightarrow{\;\;\pi_n\;\;} \Sigma_n
 $$
 forms a (general) action operad, denoted $\pi_1(\EE,\G)$. Moreover, $\pi_1(\EE,\G)$ lies over $\G$, i.e.\ the connecting maps $\{\partial_n\}_{n\geq 1}$ in {\rm (\ref{eqt:ses pi1 of Goperads})} yield a morphism $\partial\colon\pi_1(\EE,\G)\to \G$ of (general) action operads. 
 \end{prop}
\begin{proof} Let us check Condition (3) of the definition of (general) action operad for $\pi_1(\EE,\G)$. For this task, let us write down the explicit formulas for the group structure and operadic composition obtained via (\ref{eqt:pi1 of Goperads}), i.e.\ for $\gamma,\gamma'\colon [0,1]\to \EE(n)$ sending $0$ to $\epsilon_n$ and $1$ to $\epsilon_n\triangleright \G_n$, we have the concatenation product
$$
(\gamma\cdot \gamma')_{t}:=\begin{cases}
 \gamma'_{2t} & \text{\qquad if $t\in[0,\frac{1}{2}]$},\\[1mm]
 \gamma_{2t-1}\triangleright \partial_n[\gamma'] & \text{\qquad if $t\in[\frac{1}{2},1]$}
\end{cases}
$$
and, for $\beta\colon [0,1]\to \EE(m)$ satisfying the analogous boundary conditions, we have the partial composition product
$$
(\gamma \circ_i \beta)_t:=\begin{cases}
 \theta_{3t}(i) & \text{\quad if $t\in [0,\frac{1}{3}]$},\\[1mm]
\gamma_{3t-1}\circ_i\beta_{3t-1} & \text{\quad if $t\in[\frac{1}{3},\frac{2}{3}]$},\\[1mm]
\theta_{3t-2}^{-1}\big([\gamma](i)\big)\triangleright\partial_{n+m-1}[\gamma\circ_i\beta] & \text{\quad if $t\in[\frac{2}{3},1]$},
\end{cases} 
$$
where $\theta_t(i)$ is a choice of path $\epsilon_{n+m-1}\simeq \epsilon_n\circ_i\epsilon_{m}$ coming from the good $\G$-basepoint of $\EE$ and we abuse notation to denote $[\gamma](i)\equiv (\pi_n\partial_n[\gamma])(i)$. Therefore, via (\ref{eqt:pi1 of Goperads}), the task is to check that the paths $(\gamma\cdot\gamma')\circ_i(\beta\cdot\beta')$ and $(\gamma\circ_{[\gamma'](i)}\beta)\cdot(\gamma'\circ_i\beta')$ are homotopic.
On the one hand, up to reparametrization, the path $(\gamma\cdot\gamma')\circ_i(\beta\cdot\beta')$ is given by:
$$
t\longmapsto\begin{cases}
 \theta_{4t}(i) & \text{\quad if $t\in[0,\frac{1}{4}]$},\\[1mm]
\gamma'_{4t-1}\circ_i\beta'_{4t-1} & \text{\quad if $t\in[\frac{1}{4},\frac{1}{2}]$},\\[1mm]
\big(\gamma_{4t-2}\triangleright\partial_n[\gamma']\big)\circ_i\big(\beta_{4t-2}\triangleright \partial_m[\beta']\big)& \text{\quad if $t\in[\frac{1}{2},\frac{3}{4}]$},\\[1mm]
\theta^{-1}_{4t-3}\big([\gamma\cdot\gamma'](i)\big)\triangleright \partial_{n+m-1}[(\gamma\cdot\gamma')\circ_i(\beta\cdot\beta')] & \text{\quad if $t\in[\frac{3}{4},1]$}.
\end{cases} 
$$
On the other hand, up to reparametrization, the path $(\gamma\circ_{[\gamma'](i)}\beta)\cdot(\gamma'\circ_i\beta')$ is homotopic to:
$$
t\longmapsto\begin{cases}
 \theta_{4t}(i) & \text{if $t\in [0,\frac{1}{4}]$},\\[1mm]
\gamma'_{4t-1}\circ_i\beta'_{4t-1} & \text{if $t\in[\frac{1}{4},\frac{1}{2}]$},\\[1mm]
\big(\gamma_{4t-2}\circ_{[\gamma'](i)}\beta_{4t-2}\big)\triangleright \partial_{n+m-1}[\gamma'\circ_i\beta'] & \text{if $t\in[\frac{1}{2},\frac{3}{4}]$},\\[1mm]
\theta^{-1}_{4t-3}\big([\gamma]([\gamma'](i))\big)\triangleright\partial_{n+m-1}[\gamma\circ_{[\gamma'](i)}\beta]\triangleright \partial_{n+m-1}[\gamma'\circ_i\beta'] & \text{if $t\in[\frac{3}{4},1]$},
\end{cases} 
$$
where the homotopy is induced by the contraction $\theta\big([\gamma'](i)\big) \cdot \theta^{-1}\big([\gamma'](i)\big)\simeq \mathsf{constant}$ and appears at time $t=\frac{1}{2}$.

Comparing these expressions one concludes the claim. To do so, apply the following facts: 
\begin{itemize}
    \item[(a)] $\triangleright$ is an action, $u\triangleright g\triangleright g'=u\triangleright(g\cdot g') $, 
    \item [(b)] the equivariance of the partial composition products on $\EE$ with respect to the $\G$-action, 
$
(u\triangleright g)\circ_i(u'\triangleright g')=(u\circ_{g(i)}u')\triangleright (g\circ_i g'),
$
    \item[(c)] the connecting maps are homomorphisms, $\partial_s[\alpha\cdot\alpha']=\partial_s[\alpha]\cdot\partial_s[\alpha']$,
    \item[(d)] the connecting maps are compatible with partial composition products, 
$
\partial_{n+m-1}[\gamma''\circ_i\beta'']=\partial_n[\gamma'']\circ_i\partial_{m}[\beta''].
$
\end{itemize}
These last two facts about the connecting morphisms $\{\partial_n\}_{n\geq 1}$ can be deduced from their explicit definition via the unique path-lifting property. More specifically, for $t\in [\frac{1}{2},\frac{3}{4}]$ we use (b) and (d), while for $t \in [\frac{3}{4},1]$ we use first (d) and (c), then that property (3) holds for $\G$, and finally (a).

The rest of the axioms can be checked in a similar manner. Notice that (3) is the most involved assumption since it employs all the ingredients in the construction.

\end{proof}

 The next two examples were considered in \cite{Wahl}, \cite{Zhang}, \cite{Corner-Gurski} and \cite{Yau}. The third one is a general action operad, but not an action operad.

\begin{examp}\label{ex: Braid action operad}
    Consider the action operad $\G = \Sigma$, and take $\EE$ to be the little $2$-discs operad $\C_2$ with its natural right $\G$-action: 
    $$
    \begin{tikzcd}[ampersand replacement=\&]
    \C_2(n)\times \Sigma_n \ar[rr] \&\& \C_2(n) \\[-6mm]
    \big((x_1,\dots,x_n),g\big) \ar[rr, mapsto] \&\& \big(x_{g(1)},\dots,x_{g(n)}\big)
    \end{tikzcd}.
    $$
    Now, observe that the little $1$-discs operad $\C_1$ is homotopy equivalent to the symmetric operad $\mathsf{Ass}=\mathsf{As}_{\Sigma}$, and the inclusion $\C_1\to \C_2$ is a good $\G$-basepoint. The operad $\Br = \pi_1(\C_2,\Sigma)$ is the action operad of \emph{braid groups}.
\end{examp}

\begin{examp}
    The previous example remains valid if we replace the little $2$-discs operad $\C_2$ by its framed version $\C^{\fr}_2\simeq \C_2\ltimes \mathsf{SO}(2)$\footnote{Note that \cite{Wahl} denotes this operad by $ \C_2\rtimes \mathsf{SO}(2)$ instead of $\C_2\ltimes \mathsf{SO}(2)$. Our choice of notation matches the consolidated convention in group theory to denote semidirect products.
    }. The resulting operad $\RBr = \pi_1(\C^{\fr}_2,\Sigma)$ is the action operad of \emph{ribbon braid groups}. Recall that 
    $$\RBr_k\cong \Br_k\ltimes\, \mathbb{Z}^{\times k},$$
    where $\mathbb{Z}^{\times k}$ accounts for the number of full twists on each ribbon. 
\end{examp}

\begin{examp}
    Consider the general action operad $\G = \Sigma^{\pm}$, and take $\EE$ to be the framed little $2$-discs operad $\C^\fr_2$ with the following $\G$-action:
    $$
    \begin{tikzcd}[ampersand replacement=\&]
    \C_2^{\fr}(n)\times \Sigma_n^{\pm} \ar[rr] \&\& \C_2^{\fr}(n) \\[-6mm]
    \big((x_1,\dots,x_n),(g,A)\big) \ar[rr, mapsto] \&\& \big(x^A_{g(1)},\dots,x^A_{g(n)}\big)
    \end{tikzcd},
    $$
    where $x_i^A = x_i$ for any $i\notin A$ and $x_i^A$ is the precomposition of the embedding $x_i\colon D^2\to D^2$ with a $\pi$-rotation otherwise. Now, observe that the ``unoriented little $1$-discs operad'' $\C^{\mathsf{un}}_1\simeq \C_1\ltimes \OO(1)$ is homotopy equivalent to the operad $\mathsf{As}_{\Sigma^{\pm}}$ since there is a homotopy equivalence 
    $$
    \C_1^{\mathsf{un}}(k)\simeq \C_1(k)\times \OO(1)^{\times k} 
    $$ 
    and bijections
    $$
      \mathsf{As}_{\Sigma^{\pm}}(k)\cong \Sigma_{k}^{\pm}\cong \Sigma_k\times \{\pm \id\}^{\times k}.
    $$
    Hence, the inclusion $\C^{\mathsf{un}}_1\to \C^{\fr}_2$ determined by the inclusion $\C_1\to \C_2$ and the homomorphism $\OO(1)\to \mathsf{SO}(2)$, $-\id\mapsto e^{i\pi}$, is a good $\G$-basepoint. The operad $\URBr = \pi_1(\C^{\fr}_2,\Sigma^{\pm})$ is the general action operad of \emph{unoriented ribbon braid groups}. One can identify unoriented ribbon braid groups in a similar manner to the case of plain ribbon braids, i.e.\  $\URBr_{k}\cong \Br_k\ltimes\, \mathbb{Z}^{\times k}$, but now $\mathbb{Z}^{\times k}$ accounts  for the number of half-twists on each ribbon. With such a description, the canonical morphism $\URBr\to \Sigma^{\pm}$ sends $(\beta_k;n_1,\dots,n_k)$ to $(\pi(\beta_k),A_{\underline{n}})$, where $\pi(\beta_k)$ is the underlying permutation of the braid $\beta_k$ and $A_{\underline{n}}$ is given by the set of $i$'s so that $n_i$ is odd. See Figure \ref{fig:FramedUnordConfAndRibbon} for an illustration of the difference between $\RBr$ and $\URBr$.

    \begin{figure}[h]
    \centering
    \includegraphics
    {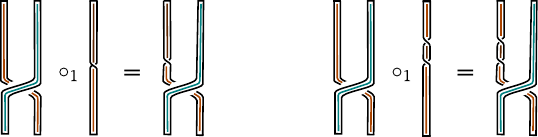}
    \caption{From left to right, the operation $\sigma\circ_1\tau=(\sigma;1,0)$ in $\URBr_2\cong \Br_2\ltimes\, \mathbb{Z}^{\times 2}$ and $\sigma\circ_1\widehat{\tau}=(\sigma;1,0)$ in $\RBr_2\cong \Br_2\ltimes\, \mathbb{Z}^{\times 2}$.
    }
    \label{fig:FramedUnordConfAndRibbon}
    \end{figure}
\end{examp}

By construction, there is a commutative diagram of general action operads
\begin{equation}\label{eq:braid_maps}
\begin{tikzcd}[ampersand replacement=\&]
   \Br \ar[rd, bend right=40]\ar[r] \& \RBr \ar[r] \ar[d] \& \URBr \ar[d] \&[-4mm] (\beta_k;n_1,\dots,n_k) \ar[r, mapsto]\ar[d,mapsto] \& (\beta_k;2n_1,\dots,2n_k)\ar[d,mapsto]\\
   \& \Sigma \ar[r] \& \Sigma^{\pm} \& \pi(\beta_k) \ar[r, mapsto] \& (\pi(\beta_k),\emptyset)
\end{tikzcd}.
\end{equation}
Note that $\RBr$ is an action operad, while $\URBr$ equipped with the obvious projection $\URBr\to \Sigma$, $(\beta_{k};n_1,\dots,n_k)\mapsto \pi(\beta_{k})$ is just a general action operad since, for example, the square
$$
\begin{tikzcd}[ampersand replacement=\&]
\URBr_1\times \URBr_{2}\ar[r,"\circ_1"] \ar[d] \& \URBr_2\ar[d]\\
\Sigma_1\times \Sigma_2 \ar[r,"\circ_1"'] \& \Sigma_2
\end{tikzcd}
$$
does not commute. For instance, take $(\tau,e_2)=\big((e_1;1),(e_2;0,0)\big)\in \URBr_1\times \URBr_2$ and follow the two directions to obtain $e_2\neq (1,2)$ in $\Sigma_2$ (see Figure \ref{fig:Composition in twisted braid}).

\begin{rem}
Under the correspondence between general action operads and operadic cloning systems, the unordered ribbon braided operad $\URBr$ corresponds to the cloning system of twisted braid groups $\{B^{twist}_n\}_n$ described in \cite[Example~4.2]{Zaremsky}. 

In \cite[3.5.3]{Th}, Thumann introduces the following braided operad $\EE$: Consider first the free braided operad generated in arity $1$ by an operation $\tau$ and in arity $2$ by an operation $\sigma$. The operad $\EE$ is defined as the result of quotienting this operad by the relation $\tau \circ_1 e_2 = (\sigma\circ_1 \tau)\circ_2 \tau$. This operad is isomorphic to the underlying operad of $\URBr$, and the braid action comes from the maps \eqref{eq:braid_maps}: the operation $\tau$ is the half-twist of a single ribbon, while the operation $\sigma$ is the generator of the braid group. The relation is depicted in Figure \ref{fig:Composition in twisted braid}.

 \begin{figure}[h]
    \centering
    \includegraphics[width=6cm]{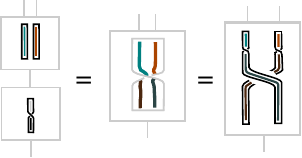}
    \caption{The relation $\tau\circ_1 e_2= (\sigma\circ_1 \tau)\circ_2\tau$ seen in the general action operad $\URBr$.}
    \label{fig:Composition in twisted braid}
    \end{figure}
\end{rem}


\section{A case study: the correspondence for braid groups
}
\label{sec:braids}
Before we embark on the proof of our main theorem in its full generality, we explain the situation with the special case of braid groups. As will become apparent, the general argument is an abstraction of the ideas introduced here, and will be treated in the next two sections.

Let $\Br=\big(\Br_{\bullet},\pi\big)$ be the collection of all braid groups equipped with the canonical projection homomorphisms $\pi_n\colon\Br_{n}\to \Sigma_n$. 

On the one hand, recall from Example~\ref{examp:cloning} the (bilateral) cloning system for braid groups $\Br=(\Br_{\bullet},\pi,\iota,\zeta,\kappa)$, where the maps $\iota$ and $\zeta$ add one strand on the right or the left, respectively, and $\kappa$ are the cloning maps given by duplicating a strand.

On the other hand, Example \ref{ex: Braid action operad} endows the collection $\Br_{\bullet}$ with an action operad structure with respect to the ``substitution maps'' $\circ_{j}\colon \Br_{n+1}\times\Br_m\to \Br_{n+m}$, which correspond to replacing the $j$-th strand of the first braid by the second braid. Figure  \ref{fig:CloningToOperad} which contains a depiction of this operation; for a more detailed discussion, see \cite[Definition 1.6 and Example 1.12(2)]{Corner-Gurski} for details.

We now explain how to obtain one structure from the other, illustrating the main ideas of the procedure with several pictures.

\subsection{From the action operad to the (bilateral) cloning system.}
The morphisms $\iota$ and the cloning maps $\kappa$ for the cloning system are obtained by using the identity $e_2\in \Br_2$ and the composition products $\circ_i$, as explained in Figures~\ref{fig:OperadToIota} and~\ref{fig:OperadToKappa}, respectively. Thus one obtains the cloning system $\Br$ described in Example~\ref{examp:cloning}. 
\begin{figure}
    \centering
    \includegraphics[width=10cm]{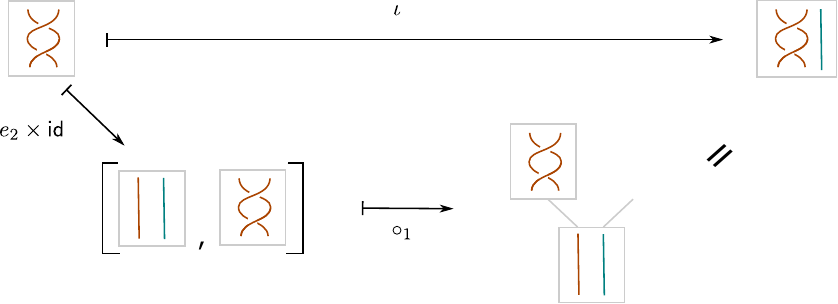}
    \caption{Constructing the map $\iota$. Insert the given braid in the first strand of $e_2$.}
    \label{fig:OperadToIota}
    \end{figure} 
        
    \begin{figure}
    \centering
    \includegraphics[width=10cm]{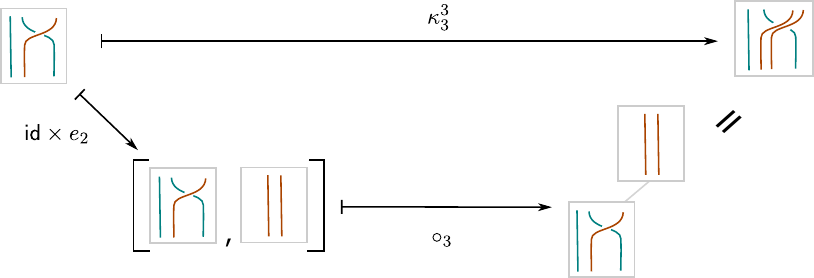}
    \caption{Constructing the map $\kappa_{j}$. Replace the $j$th strand of the given braid by $e_2$.}
    \label{fig:OperadToKappa}
    \end{figure}
In order to get a bilateral cloning system, the maps $\zeta$ are defined similarly to the maps $\iota$, but replacing the second strand of $e_2$ instead of the first one, that is, using $\circ_2$ instead of $\circ_1$; see Figure~\ref{fig:OperadToZeta}.
\begin{figure}
    \centering
    \includegraphics[width=10cm]{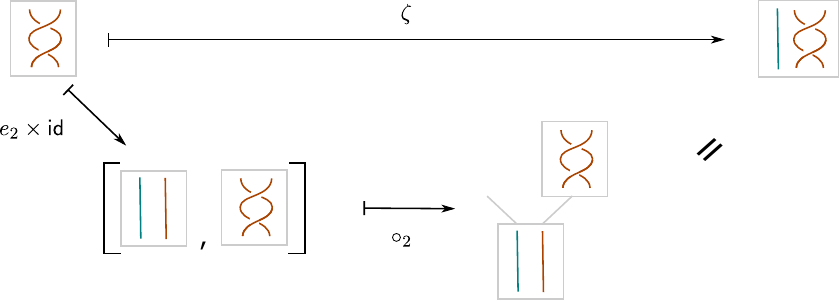}
    \caption{Constructing the map $\zeta$. Insert the given braid in the second strand of $e_2$.}
    \label{fig:OperadToZeta}
    \end{figure}    

\subsection{From the (operadic) bilateral cloning system to the action operad} We now explain how to use the bilateral cloning system structure on $\Br$ in order to build an action operad structure on $\Br_{\bullet}$. To see how the composition product maps $\circ_j$ are obtained, we proceed as follows. First, we replicate $m$ times the $j$th strand of the first braid using the cloning maps, where $m$ is the number of strands of the second braid. Then we add, by using $\iota$ and $\zeta$ strands to the right and left, respectively, of the second braid, so that it has the same number of strands as the cloned braid. Finally, we just multiply the two braids obtained this way. Figure~\ref{fig:CloningToOperad} shows an example of this construction. 
 \begin{figure}
    \centering
    \includegraphics[width=10cm]{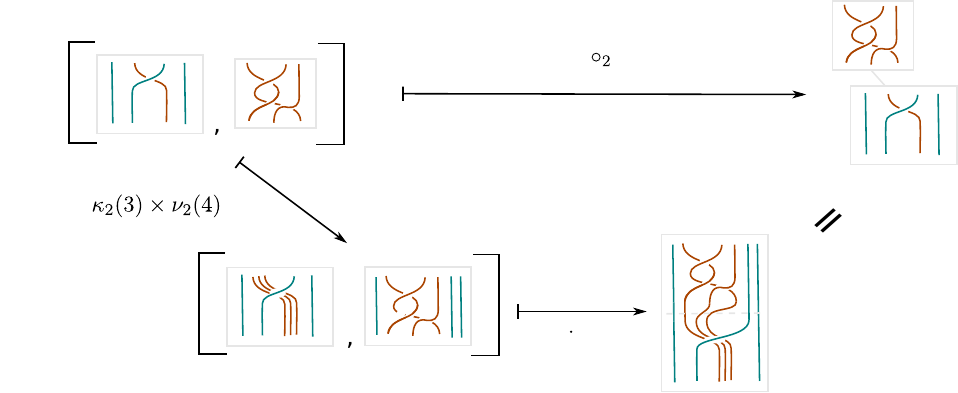}
    \caption{Constructing the substitution maps $\circ_j$ from the cloning system data.
    }
    \label{fig:CloningToOperad}
    \end{figure}

One can check that the maps $\circ_j$ built above satisfy all the properties required for  endowing $\Br_{\bullet}$ with an operad structure on sets. For instance, Figure~\ref{fig:AssociativityI} and Figure~\ref{fig:AssociativityII} show the verification of the associativity axiom for a particular example. 
 \begin{figure}
    \centering
    \includegraphics[width=7cm]{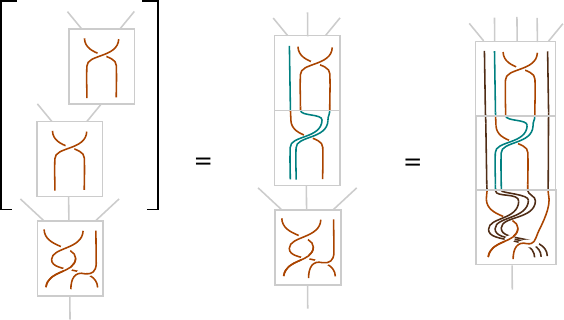}
    \caption{First half of the associativity axiom. The colors represent the application of $\iota$, $\zeta$ and $\kappa$ in each step.}
    \label{fig:AssociativityI}
    \end{figure}
    \begin{figure}
    \centering
    \includegraphics[width=7cm]{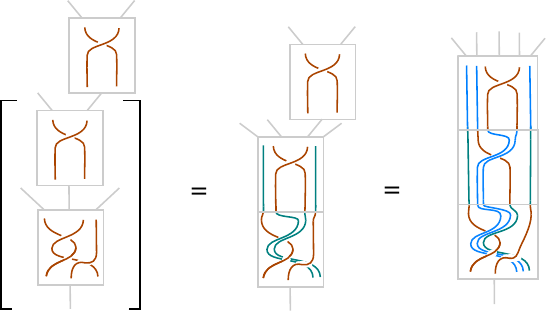}
    \caption{Second half of the associativity axiom.
    }
    \label{fig:AssociativityII}
    \end{figure}
    
Finally, a simple computation, depicted in Figure~\ref{fig:Compatibility}, shows that the maps $\circ_j$ are compatible with the underlying group structure of the braid groups. The canonical projection maps $\pi_n\colon \Br_n\to \Sigma_n$ respect operadic and group structures, e.g.\ compare Figure \ref{fig:CompositionProduct Sigma} and Figure \ref{fig:CloningToOperad}. Thus, we obtain that $(\Br_{\bullet},\pi, \{\circ_i\}_i, e_1)$ is indeed an action operad. 
 \begin{figure}
    \centering
    \includegraphics[width=11cm]{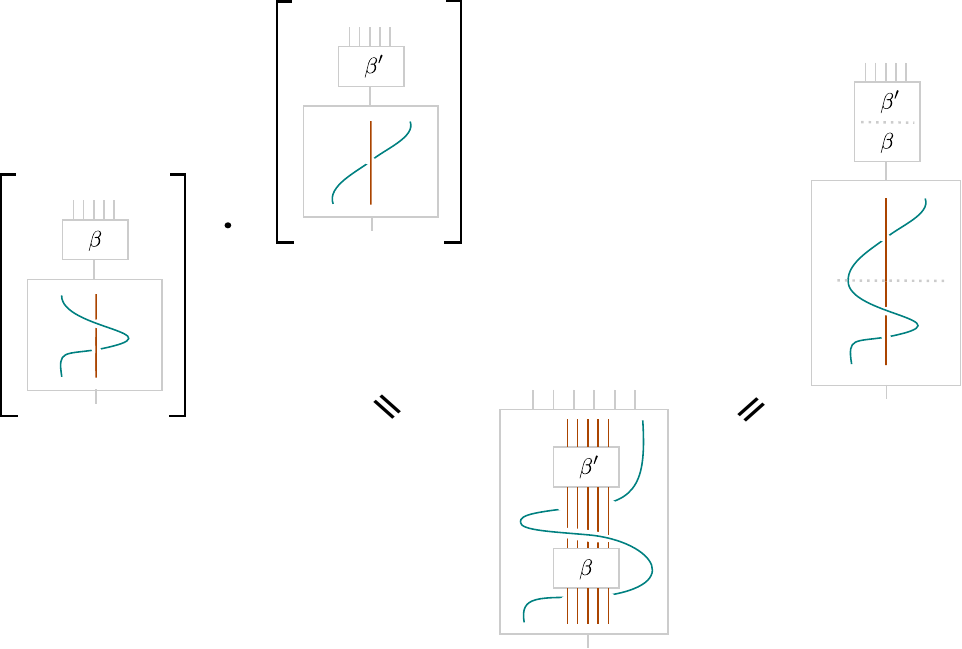}
    \caption{Compatibility axiom of the action operad.}
    \label{fig:Compatibility}
    \end{figure}

\section{From action operads to cloning systems}
\label{sec:operadtocloning}
In this section, we prove one of the implications of Theorem \ref{thm:main2} in  full generality; more concretely, we explain how to obtain a (restricted) operadic cloning system from an arbitrary action operad.

Let $\G=(\G_{\bullet},\pi, \{\circ_i\}_i, \id)$ be an action operad, and
 denote by $e_n\in \G_n$ the identity element of  $\G_n$. If $g\in \G_n$ and $i\in \{1,\ldots,n\}$, we denote by $g(i) = \pi_n(g)(i)$. From the action operad structure, one can define the following maps
$$
\kappa^n_j,\iota_n,\zeta_n\colon \G_n\longrightarrow \G_{n+1}
$$
by setting
$$
\kappa^n_j(g) = g\circ_j e_2,\quad
\zeta_n(g) = e_2\circ_2 g\quad\mbox{ and }\quad
\iota_n(g) = e_2\circ_1 g.
$$
The following observation will be useful in what follows:
\begin{rem}\label{rem:properties} Note that in an action operad we always have that $e_n\circ_i e_m = e_{n+m-1}$ because $e_n\circ_i e_m = (e_n\cdot e_n) \circ_i (e_m\cdot e_m) = (e_n\circ_i e_m)\cdot (e_n\circ_i e_m)$. Moreover, the following hold
$$
    \nu_j^n(m)(g) = e_m\circ_j g\quad\mbox{ and }\quad
    \kappa_j^n(m)(g) = g\circ_j e_m,
$$
where $\kappa_j^n(m)$ and $\nu_j^n(m)$ are defined as in Definition~\ref{defn:OperadicCloningSystem}. Notice that, from the first identity, one deduces that
$$
    \iota_n(r)(g)=e_{r+1}\circ_1 g \quad \text{ and }\quad \zeta_n(l)(g)=e_{l+1}\circ_{l+1}g.
$$
\end{rem}

\begin{thm}\label{thm:OperadtoCS} 
Let $\G=(\G_{\bullet},\pi, \{\circ_i\}_i, \id)$ be an action operad. Then, the quintuple $(\G_{\bullet}, \iota, \zeta, \kappa,\pi)$, with $\iota$, $\zeta$ and $\kappa$ as defined above, is a restricted operadic  cloning system. If $\G=(\G_{\bullet},\pi, \{\circ_i\}_i, \id)$ is a general action operad, then $(\G_{\bullet}, \iota, \zeta, \kappa,\pi)$ is an operadic cloning system.
\end{thm}
\begin{proof}
We will prove it in several steps. We will make essential use of  the fact that~$\G$ is an action operad, the definition of the morphisms $\iota$, $\zeta$ and $\kappa$, and Remark~\ref{rem:properties}. First, we prove that the maps $\iota$ and $\zeta$ are homomorphisms; indeed,
\[\iota(f\cdot g) = e_2\circ_1 (f\cdot g) =
(e_2\cdot e_2)\circ_1(f\cdot g)=
(e_2\circ_{1} f)\cdot (e_2\circ_1 g) = \iota(f)\cdot \iota(g),\]
\[\zeta(f\cdot g) = e_2\circ_2 (f\cdot g) =
(e_2\cdot e_2)\circ_2(f\cdot g)=
(e_2\circ_{2} f)\cdot (e_2\circ_2 g) = \zeta(f)\cdot \zeta(g).\]
We now show that the quadruple $(\G_\bullet,\iota,\kappa,\pi)$ is a cloning system by checking that it satisfies properties (i)--(v) of Definition~\ref{defn:CloningSystem}.
\begin{itemize}
\item[(i)] $\pi_{n+1}\circ \iota_n = \lambda_n\circ\pi_n$. For every $g\in \G_n$ we have that
\begin{align*}
\pi_{n+1}(\iota_n(g)) = \pi_{n+1}(e_2\circ_1 g) = \pi_{2}(e_2)\circ_1 \pi_{n}(g) = e_2\circ_1 \pi_{n}(g) = \lambda_n(\pi_n(g)).
\end{align*}
\item[(ii+)] $\pi_{n+1}\circ \kappa^n_j = c_j^n\circ\pi_n$. For every $g\in\G_n$ and all $1\le j\le n$ we have that
\begin{align*}
\pi_{n+1}(\kappa^n_j (g)) = \pi_{n+1}(g\circ_j e_2) = \pi_{n}(g)\circ_j \pi_{2}(e_2) = \pi_{n}(g)\circ_j e_2 = c^n_j(\pi_n(g)).
\end{align*}
\item[(iii)] $\iota_{n+1}\circ \kappa_j^n = \kappa^{n+1}_j\circ\iota_n$. For every $g\in\G_n$ and all $1\le j\le n$ we have that
\begin{align*}
\iota_{n+1}\circ \kappa_j^n(g) &= \iota_{n+1}(g\circ_j e_2) = e_2\circ_1(g\circ_j e_2)  \\&=  ( e_2\circ_1 g)\circ_j e_2 =  \iota_n(g)\circ_j e_2 = \kappa^{n+1}_j(\iota_n(g)).\end{align*}
\item[(iv)] $\kappa^{n+1}_l\circ\kappa_j^n = \kappa_{j+1}^{n+1}\circ \kappa_l^n$. For every $g\in\G_n$ and all $l< j\le n$ we have that
\begin{align*}
\kappa^{n+1}_l(\kappa^n_j(g)) &= \kappa^{n+1}_l(g\circ_j e_2) = (g\circ_j e_2) \circ_l e_2 \\
&=(g\circ_l e_2) \circ_{j+1} e_2 = \kappa^n_l(g) \circ_{j+1} e_2 = \kappa^{n+1}_{j+1}(\kappa^n_l(g)).
\end{align*}
\item[(v)] $\kappa_j^n(g\cdot h) = \kappa^n_{h(j)}(g)\cdot \kappa_j^n(h)$ for every $g,h \in \G_n$. For all $j\le n$ we have that
\begin{align*}
\kappa_j^n(g\cdot h) &= (g\cdot h)\circ_j e_2 = (g\cdot h)\circ_j (e_2\cdot e_2) \\
&= (g\circ_{h(j)} e_2)\cdot (h\circ_j e_2)
= \kappa^n_{h(j)}(g)\cdot \kappa_j^n(h)
\end{align*}
\end{itemize}
Finally, we prove that $(\G_{\bullet}, \iota, \zeta, \kappa,\pi)$ is a restricted operadic cloning system by checking properties (i'), (iii'), (vi), (vii), (vii'), (iv+) of Definition~\ref{defn:NCCloningSystem}, and (viii), (ix) of Definition~\ref{defn:OperadicCloningSystem}. 
First, properties (i'), (iii') and (vii') are proved as properties (i), (iii) and (vii) simply by replacing $\circ_1$ by $\circ_2$ and $\iota$ by $\zeta$.
\begin{itemize}
\item[(vi)] $\zeta_{n+1}\circ \iota_n = \iota_{n+1}\circ\zeta_{n}$. For every $g\in \G_n$ and all $n\ge 1$ we have that
\begin{align*}
    \zeta_{n+1}(\iota_n(g)) &= e_2\circ_2(e_2\circ_1 g) = (e_2\circ_2 e_2)\circ_2 g = e_3\circ_2 g \\&= (e_2\circ_1 e_2)\circ_2 g = e_2\circ_1 (e_2\circ_2 g) = \iota_{n+1}(\zeta_{n}(g)).
\end{align*}
\item[(vii)] $\iota_{n+1}\circ\iota_n = \kappa^{n+1}_{n+1}\circ \iota_n$. For every $g\in \G_n$ and all $n\ge 1$ we have that
\begin{align*}
\iota_{n+1}(\iota_{n}(g)) &= e_2\circ_1(e_2\circ_1 g) = (e_2\circ_1 e_2)\circ_1 g = e_3\circ_1 g, \\
\kappa^{n+1}_{n+1}(\iota_n(g)) &= (e_2\circ_1 g)\circ_{n+1} e_2 = (e_2\circ_2 e_2)\circ_1 g = e_3\circ_1 g.
\end{align*}
\item[(iv+)] $\kappa^{n+1}_{j+1}\circ \kappa^n_{j} = \kappa_{j}^{n+1}\circ \kappa_j^n$. For all $g\in \G_n$ and all $1\le j\le n$ we have that 
\begin{align*}
\kappa^{n+1}_{j+1}(\kappa^n_{j}(g)) &=
(g\circ_j e_2)\circ_{j+1} e_2 = g\circ_j(e_2\circ_2 e_2) = g\circ_2 e_3  \\ &= g\circ_j (e_2\circ_1 e_2) = (g\circ_j e_2)\circ_j e_2 = \kappa_{j}^{n+1}(\kappa_j^n(g)).
\end{align*}
\item[(viii)] $\kappa_j^n(m)(g)\cdot \nu^m_j(n)(h) = \nu^m_{g(j)}(n)(h) \cdot \kappa_j^n(m)(g)$ for every $g\in\G_n$ and $h\in\G_m$. For all $m,n\geq 1$ we have that
\begin{align*}
\kappa_j^n(m)(g)\cdot \nu^m_j(n)(h) &=
(g\circ_j e_m)\cdot (e_n\circ_j h) \\
&= (g\cdot e_n)\circ_{j} (e_m\cdot h) \\
&= (e_n\cdot g)\circ_{j} (h\cdot e_m) \\
&= (e_n\circ_{g(j)} h)\cdot (g\circ_{j} e_m) \\
&= \nu^m_{g(j)}(n)(h)\cdot \kappa_{j}^n(m)(g).
\end{align*}

\item[(ix)] $\iota_n(m)(g)\cdot \zeta_m(n)(h) = \zeta_m(n)(h)\cdot \iota_n(m)(g)$ for every $g\in\G_n$ and $h\in\G_m$. For all $m,n\ge 1$ we have that
\begin{align*}
    \iota_n(m)(g)\cdot \zeta_m(n)(h) &= (e_{m+1}\circ_{1} g)\cdot (e_{n+1}\circ_{n+1} h) \\
    &= (e_{m+1}\circ_1 g)\cdot((e_2\circ_1 e_n)\circ_{n+1} h) \\
    &= (e_{m+1}\circ_1 g)\cdot((e_2\circ_2 h)\circ_1 e_n) \\
    &= (e_{m+1}\cdot (e_2\circ_2 h))\circ_1 (g\cdot e_n) \\
    &= ((e_2\circ_2 h)\cdot e_{m+1})\circ_1 (e_n\cdot g) \\
    &= ((e_2\circ_2 h)\cdot (e_{2}\circ_2 e_m))\circ_1 (e_n\cdot g) \\
    &= ((e_2\circ_2 h)\circ_1 e_n)\cdot ((e_2\circ_2e_m)\circ_1 g) \\
    &= ((e_2\circ_1 e_n)\circ_{n+1} h)\cdot ((e_{m+1})\circ_1 g) \\
    &= ((e_{n+1})\circ_{n+1} h)\cdot ((e_{m+1})\circ_1 g) \\
    &= \zeta_m(n)(h)\cdot \iota_n(m)(g).
\end{align*}
\end{itemize}
Therefore, $(\G_{\bullet}, \iota, \zeta, \kappa,\pi)$ is a restricted operadic cloning system as we wanted to show.

If $(\G_{\bullet}, \iota, \zeta, \kappa,\pi,\id)$ were a general action operad, then the proof of Condition (ii+) would restrict instead to a proof of Condition (ii).
\end{proof}

\section{From cloning systems to action operads}\label{sect:FromCStoActionOperads}
We now explain how to construct (general) action operads from an operadic cloning system. Let $\G=(\G_{\bullet},\iota,\zeta,\kappa,\pi)$ be an operadic cloning system. The following identities can be derived from the identities in the definition of an operadic cloning system (recall from Notation~\ref{notat: generalized structure maps of a bilateral cloning system} the construction of the maps $\kappa_i^n(m)$ and $\nu_i^n(m)$):
\begin{align}\label{eqt:6.1}
  \kappa_i^{n+m-1}(l)\circ \kappa_j^n(m) &=
\begin{cases}
    \kappa_{j}^n(m+l-1) & \text{if $j\leq i < j+m$}, \\[1mm]
    \kappa_{j+l-1}^{n+l-1}(m)\circ\kappa_i^n(l) & \text{if $i<j$}.
\end{cases}  
\\[1em]
\label{eqt:6.2}
    \kappa_j^{n+m-1}(l)\circ \nu^n_i(m) &= \begin{cases}
		\nu^{n+l-1}_i(m)\circ \kappa_{j-i+1}^n(l) & \text{if $i\leq j< i+n$},\\[1mm]
		\nu^{n}_{i+l-1}(m+l-1) & \text{if $j< i$}, \\[1mm]
		\nu^{n}_i(m+l-1) & \text{if $j\geq i+n$},
		\end{cases}
\\[1em]
\label{eqt:6.3}
    \nu_j^{n+m-1}(l)\circ \nu_i^n(m)& =\nu_{j+i-1}^n(m+l-1),
\\[1em]
\label{eqt:6.4}
    \kappa_j^n(m)(f\cdot g) &= \kappa_{g(j)}^n(m)(f)\cdot \kappa_j^n(m)(g) \qquad \mbox{if } 1\leq j\leq n,
\end{align}
\begin{multline}\label{eqt:6.5}
        \nu_j^l(n+m-1)(g)\cdot \nu_{i+l-1}^m(n+l-1)(h) \\ =  \nu_{i+l-1}^m(n+l-1)(h)\cdot \nu_j^l(n+m-1)(g)\qquad \mbox{if $j<i$}
\end{multline}

Note that $\nu^n_i(m)$ is a group homomorphism, that $\nu^n_i(m)(g)$ acts trivially on all $j<i$ and all $j\geq i+m$ through $\pi_{n+m-1}$ and that $\kappa^n_i(m)(e_n) = e_{n+m-1}$, because
\[\kappa^n_i(e_n) = \kappa^n_i(e_n\cdot e_n) = \kappa^n_i(e_n)\cdot \kappa^n_i(e_n)\]
and therefore $\kappa^n_i(e_n) = e_{n+1}$.

\begin{rem} In order to deduce the previous identities one uses:
\begin{align*}
(\ref{eqt:6.1})&\leftrightsquigarrow \text{(iv)/(iv+),} && & (\ref{eqt:6.3})&\leftrightsquigarrow \text{(vi),}\\
&& (\ref{eqt:6.2})&\leftrightsquigarrow \text{(iii)/(iii'), (vii)/(vii') and (vi),} 
\\
 (\ref{eqt:6.4})&\leftrightsquigarrow \text{(v),} & &&(\ref{eqt:6.5})&\leftrightsquigarrow \text{(ix).}
\end{align*}
\end{rem}

In order to define an action operad structure on $\G_{\bullet}$ and since we already have a map $\pi\colon \G_{\bullet}\to \Sigma_{\bullet}$, it is enough to define the partial composition products $\{\circ_i\}_i$, which we do as follows. The map
\[
\circ_i\colon \G_n\times \G_m\longrightarrow \G_{n+m-1}
\]
is the following composition
\[\begin{tikzcd}
\G_{n}\times \G_{m}\ar[r,"\kappa_{i}^n(m)\times \nu_i^m(n)"] &[15mm]
\G_{n+m-1}\times\G_{n+m-1} \ar[r,"\cdot"] &
\G_{n+m-1}
\end{tikzcd},\]
that is, $f\circ_i g := \kappa_i^n(m)(f)\cdot \nu^m_i(n)(g)$, for every $f\in\G_n$ and $g\in \G_m$; see Figure \ref{fig:CloningToOperad}. We also set $\id := e_1\in \G_1$.

\begin{prop}\label{prop:CloningtoOper}
    Let $\G=(\G_{\bullet}, \iota, \zeta, \kappa,\pi)$ be a restricted operadic cloning system. Then the quadruple $(\G_{\bullet},\pi, \{\circ_i\}_i, \id)$, with $\{\circ_i\}_i$ as defined above is an action operad. If the restricted condition is dropped, then we obtain a general action operad.
\end{prop}
\begin{proof} First, we check that $e_1$ is indeed a unit for the partial composition product. For every $f\in\G_n$ and every $g\in \G_m$, use the observations below (\ref{eqt:6.5}) to obtain
\begin{align*}
    f\circ_i e_1 &:= \kappa^n_i(1)(f)\cdot \nu_i^1(n)(e_1) = f\cdot e_n = f, \\
    e_1\circ_1 g &:= \kappa^1_1(m)(e_1) \cdot \nu_1^m(1)(g) = e_m\cdot g = g.
\end{align*}
Second, we show the associativity for the partial composition product. Let $f\in \G_n$, $g\in \G_m$ and $h\in \G_l$. Then, using Condition (v) of the definition of cloning system, 
    \begin{align}\notag
        &(f\circ_i g)\circ_j h \\ \notag &:=(\kappa_i^n(m)(f)\cdot \nu^m_i(n)(g))\circ_j h \\ \notag
&:= \kappa^{n+m-1}_j(l)\Big(\kappa_i^n(m)(f)\cdot \nu_i^m(n)(g)\Big)\cdot \nu_j^l(n+m-1)(h) \\ \label{eq:309}
&=
\Big(\kappa_{g^*(j)}^{n+m-1}(l)(\kappa_i^n(m)(f))\Big)\cdot \Big(\kappa_{j}^{n+m-1}(l)(\nu_i^m(n)(g))\Big)\cdot \Big(\nu_j^l(n+m-1)(h)\Big), 
\end{align}
where we denote $g^*=\nu_i^m(n)(g)$. Depending on the indices $i$ and $j$, we have to consider the following three cases.

\bigskip
\noindent {\it Case 1.} Suppose first that $j<i$. In this case $g^*(j) = j$ and we have that 
\begin{align*}
    \kappa_j^{n+m-1}(l)\circ \kappa_i^n(m) &= \kappa_{i+l-1}^{n+l-1}(m)\circ \kappa_j^n(l),\\[2mm]
    \kappa_j^{n+m-1}(l)\circ \nu_i^m(n) &= \nu_{i+l-1}^m(n+l-1)
\end{align*}
by (\ref{eqt:6.1}) and (\ref{eqt:6.2}).
Then \eqref{eq:309} becomes
\[\Big(\kappa_{i+l-1}^{n+l-1}(m)(\kappa_j^n(l)(f))\Big)\cdot \Big(\nu_{i+l-1}^m(n+l-1)(g)\Big)\cdot \Big(\nu_j^l(n+m-1)(h)\Big). \]
On the other hand, again by (v), we have that
    \begin{align}\notag 
        &(f\circ_j h)\circ_{i+l-1} g  \\ \notag
&:=(\kappa_j^n(l)(f)\cdot \nu_j^l(n)(h))\circ_{i+l-1} g \\ \notag      
&:= \kappa_{i+l-1}^{n+l-1}(m)\Big(\kappa_j^n(l)(f)\cdot \nu_j^l(n)(h)\Big)\cdot \nu_{i+l-1}^m(n+l-1)(g) \\ \label{eq:310}
&=
\Big(\kappa_{h^*(i+l-1)}^{n+l-1}(m)(\kappa_j^n(l)(f))\Big)\cdot \Big(\kappa_{i+l-1}^{n+l-1}(m)(\nu_j^l(n)(h))\Big)\cdot \Big(\nu_{i+l-1}^m(n+l-1)(g)\Big),
\end{align}
where $h^*=\nu_j^l(n)(h)$. Now, \eqref{eq:309} and \eqref{eq:310} are equal by (\ref{eqt:6.5}) and since 
$$
h^*(i+l-1) = \nu_j^l(n)(h)(i+l-1) = i+l-1\quad\mbox{and}
$$
$$
\kappa_{i+l-1}^{n+l-1}(m)(\nu_j^l(n)(h)) = \nu_j^l(n+m-1)(h).
$$
The last equalities hold by (\ref{eqt:6.2}) because $j<i$ and therefore $i+l-1\ge j+l$.

\bigskip
\noindent{\it Case 2.} Suppose now that $j\geq i+m$. Again $g^*(j)=j$ and we have that
\begin{align*}
    \kappa_j^{n+m-1}(l)\circ \kappa_i^n(m) &= \kappa_{i}^{n+l-1}(m)\circ \kappa_{j-m+1}^n(l),\\[2mm]
   \kappa_j^{n+m-1}(l)\circ \nu_i^m(n) &= \nu_i^m(n+l-1)
\end{align*}
by (\ref{eqt:6.1}) and (\ref{eqt:6.2}). The first equality holds because $j\ge i+m$ and therefore $i<j-m+1$.
Then \eqref{eq:309} becomes
\[\Big(\kappa_{i}^{n+l-1}(m)(\kappa_{j-m+1}^n(l)(f))\Big)\cdot \Big(\nu_i^m(n+l-1)(g)\Big)\cdot \Big(\nu_j^l(n+m-1)(h)\Big). \]
On the other hand, using (v), we get
    \begin{align}\notag 
        &(f\circ_{j-m+1} h)\circ_{i} g \\ \notag
&:= (\kappa_{j-m+1}^n(l)(f)\cdot \nu_{j-m+1}^l(n)(h))\circ_i g \\ \notag  &:= \kappa_{i}^{n+l-1}(m)\Big(\kappa_{j-m+1}^n(l)(f)\cdot \nu_{j-m+1}^l(n)(h)\Big)\cdot \nu_i^m(n+l-1)(g) \\ \label{eq:311}
&=
\Big(\kappa_{h^*(i)}^{n+l-1}(m)(\kappa_{j-m+1}^n(l)(f))\Big)\!\cdot \!\Big(\kappa_{i}^{n+l-1}(m)(\nu_{j-m+1}^l(n)(h))\Big)\!\cdot \!\Big(\nu_i^m(n+l-1)(g)\Big),
\end{align}
where $h^*=\nu_{j-m+1}^l(n)(h)$. But \eqref{eq:309} and \eqref{eq:311} are equal by (\ref{eqt:6.5}) since $h^*(i) = i$ and
$$
\kappa_i^{n+l-1}(m)(\nu_{j-m+1}^l(n)(h)) = \nu_j^l(n+m-1)(h).
$$
The last equality holds by (\ref{eqt:6.2}) because $j\ge i+m$ and therefore $i+m-1<j$.

\bigskip
\noindent {\it Case 3.} Suppose that $i\leq j<i+m$. In this case, $i\leq g^*(j)<i+m$ and we have that
\begin{align*}
\kappa_{g^*(j)}^{n+m-1}(l)\circ \kappa_i^n(m) &= \kappa_i^n(m+l-1)
\\
\kappa_j^{n+m-1}(l)\circ \nu_i^m(n) &= \nu_i^{m+l-1}(n)\circ \kappa_{j-i+1}^m(l).
\end{align*}
by (\ref{eqt:6.1}) and (\ref{eqt:6.2}). Then \eqref{eq:309} becomes
\[\Big(\kappa_i^n(m+l-1)(f)\Big)\cdot \Big(\nu_i^{m+l-1}(n)( \kappa_{j-i+1}^m(l)(g))\Big)\cdot \Big(\nu_j^l(n+m-1)(h)\Big).
\]
On the other hand, using (v), we obtain
\begin{align} \notag
&f\circ_i(g\circ_{j-i+1} h) \\ \notag
&:=f\circ_i (\kappa_{j-i+1}^m(l)(g)\cdot \nu_{j-i+1}^l(m)(h)) \\ \notag
&:= \kappa_i^n(m+l-1)(f)\cdot \nu_i^{m+l-1}(n)\Big(\kappa_{j-i+1}^m(l)(g)\cdot \nu_{j-i+1}^l(m)(h)\Big)\\
&=\Big(\kappa_i^n(m+l-1)(f)\Big)\cdot \Big(\nu_i^{m+l-1}(n)(\kappa_{j-i+1}^m(l)(g))\Big)\cdot \Big(\nu_i^{m+l-1}(n)(\nu_{j-i+1}^l(m)(h))\Big).\label{eq:312}
\end{align}
Again \eqref{eq:309} and \eqref{eq:312} are equal since, by (\ref{eqt:6.3}), 
$$
\nu_i^{m+l-1}(n)(\nu_{j-i+1}^l(m)(h))=\nu_j^l(n+m-1)(h).
$$

Regarding Condition (2'), it is clear that from axiom (ii) we obtain a levelwise homomorphism $\pi$ to the family of symmetric groups as in the definition of general action operad. From axiom (ii+) we obtain an operad map $\pi$ fulfilling Condition (\ref{cond:ac:2}) in the definition of action operad.

Finally, in order to prove the product rule for the action operad, i.e.\ Condition (\ref{cond:ac:3}). Let $f,f'\in G_n$ and $g,g'\in G_m$. Then we have that
\begin{align*}
(f\circ_{f'(i)} g)\cdot (f'\circ_i g')
&:= \Big(\kappa_{f'(i)}^n(m)(f)\cdot \nu_{f'(i)}^m(n)(g)\Big)\cdot \Big(\kappa_i^n(m)(f')\cdot \nu_i^m(n)(g')\Big) \\
&= \Big(\kappa_{f'(i)}^n(m)(f)\cdot  \kappa_i^n(m)(f')\Big)\cdot \Big(\nu_{i}^m(n)(g)\cdot \nu_i^m(n)(g')\Big) \\
&= \Big(\kappa_{i}^n(m)(f\cdot f')\Big)\cdot \Big(\nu_{i}^m(n)(g\cdot g')\Big) \\
&=(f\cdot f')\circ_i (g\cdot g'),
\end{align*}
where we have applied: Condition (viii) in the definition of an operadic cloning system (second line), and  (\ref{eqt:6.4}) together with the fact that $\nu_i^m(n)$ is a group homomorphism (third line).
\end{proof}

\begin{thm}\label{thm:BijectionActOpdsAndOperadicCS}
There is an explicit bijective correspondence between action operads and restricted operadic cloning systems. There is an explicit bijective correspondence between general action operads and operadic cloning systems.      
\end{thm}
\begin{proof} The constructions given in Theorem~\ref{thm:OperadtoCS} and Proposition~\ref{prop:CloningtoOper} are inverses of each other. Let $(\G_\bullet,\iota,\zeta,\kappa,\pi)$ be a restricted operadic cloning system, and let $(\hat{\G}_\bullet,\hat{\iota},\hat{\zeta},\hat{\kappa},\hat{\pi})$ be the restricted operadic cloning system that arises from the action operad associated to it. It is clear that $\hat{\G}_\bullet = \G_\bullet$ and that $\hat{\pi} = \pi$. For the maps $\iota$ we have that
\begin{align*}
\hat{\iota}_n(g) &:= e_2\circ_1 g = \kappa_1^2(n)(e_2)\cdot \nu^n_1(2)(g) = e_{n+1}\cdot \iota_n(1)(g) = \iota_n(g),
\end{align*}
and in the same way, we obtain that $\hat{\zeta}_n(g)=\zeta_n(g)$. For the cloning maps, we have that
\begin{align*}
\hat{\kappa}_j^n(g) &:= g\circ_j e_2 = \kappa_j^n(2)(g)\cdot \nu_j^2(n)(e_2) = \kappa_j^n(2)(g)\cdot e_{n+1} = \kappa_j^n(g).
\end{align*}

Conversely, let $(\hat{\G}_\bullet,\hat{\pi},\{\hat{\circ}_i\}_i,\hat{\id})$ be the action operad associated to the restricted operadic cloning system obtained from an action operad $(\G_\bullet,\pi,\{\circ_i\}_i,\id)$. It is clear that $\hat{\G}_\bullet = \G_\bullet$, $\hat{\pi}=\pi$ and $\hat{\id}=\id$. Regarding the partial composition products, let $f\in \G_n$ and $g\in\G_m$. Then we have that
\begin{align*}
f\operatorname*{\hat{\circ}}\nolimits_i g &:= \kappa_i^n(m)(f)\cdot \nu^m_i(n)(g) \\ 
&= (f\circ_i e_m)\cdot (e_n\circ_i g) \\
&= (f\cdot e_n)\circ_i (e_m \cdot g) \\
&= f\circ_i g,
\end{align*}
where we have used Remark~\ref{rem:properties} for the second equality and the product rule of the action operad for the third. 

The second statement is proven analogously.
\end{proof}

\begin{rem} As mentioned in Section \ref{section:cloning_systems}, most of the known examples of cloning systems are bilateral and even operadic. Consequently, for all those examples, we have identified a general action operad structure on them. However, a natural question arises: is it possible to interpret the remaining Example~\ref{examp:DirectPowers} and Example~\ref{examp:UpperTriangularMatrices}, in operadic terms? The answer is yes. Both examples have in common that their structural map $\pi$ is trivial (see Remark \ref{rem: pi trivial for action operad}) and that the only condition that fails for them to be operadic bilateral cloning systems is axiom (viii).

Repeating the construction of Proposition \ref{prop:CloningtoOper} for these examples, and noticing that axiom (viii) is only applied to show the compatibility of the group multiplication with the $\circ_i$ products, one gets the structure of a non-symmetric operad \emph{on sets} (see Definition \ref{def:actionoperad}). The reason why they are not operadic is that an operadic cloning system with $\pi$ trivial has an associated non-symmetric operad \emph{on groups} via Theorem \ref{thm:BijectionActOpdsAndOperadicCS}.
\end{rem}

\section{Cloning systems as crossed groups}\label{section:crossed}

Action operads have a close relationship with crossed simplicial groups \cite{Zhang, Yoshida}. In this section we review this relationship, and explain the relationship with cloning systems. As we will see, cloning systems satisfying (iv+) can be interpreted as ``crossed interval groups'', while cloning systems satisfying (iv+) without maps $\iota$ are the same as crossed simplicial groups.

Let $(\G_{\bullet},\pi,\iota,\kappa)$ be a cloning system. If we forget all the structure except the cloning maps, we can interpret the pair $(\G_\bullet,\kappa)$ as a diagram
$$
\begin{tikzcd}[ampersand replacement=\&]
\G_1
\ar[r, "\kappa_1" description] \& \G_2
\ar[r, "\kappa_1" description, shift left=2]\ar[r, "\kappa_2" description, shift right = 2] \& \G_3 
\ar[r, "\kappa_1" description, shift left=4]\ar[r, "\kappa_2" description]\ar[r, "\kappa_3" description, shift right=4] \& \;\cdots 
\end{tikzcd}
$$
which resembles the diagram representing a simplicial set $X$, but considering only the degeneracy operations
$$
\begin{tikzcd}[ampersand replacement=\&]
X[0] \ar[r, "s_0" description] \& X[1] \ar[r, "s_0" description, shift left=2]\ar[r, "s_1"' description, shift right=2] \& X[2] \ar[r, "s_0" description, shift left=3]\ar[r, "s_1" description] \ar[r, "s_2"' description, shift right=3] \& X[3] \ar[r, "s_0"description, shift left=5]\ar[r, "s_1" description, shift left=2]\ar[r, "s_2" description, shift right=2]\ar[r, "s_3"' description,shift right=5] \& X[4] \ar[r, "s_0" description,shift left=6]\ar[r, "s_1" description, shift left=3]\ar[r, "s_2" description]\ar[r, "s_3" description, shift right=3]\ar[r, "s_4"' description, shift right=6] \& \;\cdots 
\end{tikzcd}
$$
Let us formalise this viewpoint. Let $[n]=\{0,1,\ldots,n\}$ be the finite ordinal of cardinality $n+1$.
\begin{defn}
    The \emph{simplicial category} $\Delta$ has objects the non-empty finite ordinals and morphisms the order-preserving maps between them. There are two distinguished families of morphisms
\begin{align*}
    \delta^n_j\colon &[n-1]\longrightarrow [n], & \sigma^n_j\colon &[n+1]\longrightarrow [n], & 0\leq j\leq n,
\end{align*}
called \emph{cofaces} and \emph{codegeneracies}, respectively, and defined as
\begin{align*}
    \delta_j^n(i)&=\begin{cases}
        i & \text{if $i<j$} \\
        i+1 & \text{if $i\geq j$}
    \end{cases}
&
    \sigma_j^n(i)&=\begin{cases} i &\text{if $i\leq j$} \\ i-1 & \text{if $i>j$}
\end{cases}
\end{align*}
These morphisms satisfy the following relations, called \emph{cosimplicial identities}

\begin{align*}\label{eq:007}
\delta_j\circ \delta_i &= \delta_i\circ \delta_{j-1} \qquad i<j \\
\sigma_j\circ \delta_i &= \begin{cases}
    \delta_i\circ \sigma_{j-1} & i<j \\
    \id & i=j,j+1 \\
    \delta_{i-1}\circ \sigma_j & i>j+1
\end{cases}
\\
\sigma_j\circ \sigma_i &= \sigma_{j-1}\circ \sigma_i \qquad i\leq j
\end{align*}
and generate all the morphisms in the category in the sense that every morphism can be expressed as a composite of cofaces and codegeneracies. 
\end{defn}
\begin{defn}
Let $\Delta_{\surj}\subset \Delta$ be the category of non-empty finite ordinals with order-preserving surjections between them. It is the subcategory of $\Delta$ generated by the maps $\sigma_j^n$.
\end{defn}
\begin{defn}
    Let $\Set$ be the category of sets and functions. A~\emph{simplicial set} $X$ is a functor $X\colon\Delta^\op\to \Set$. The maps $X(\delta^n_i)$ are called \emph{faces} and denoted $d^n_i$ and the maps $X(\sigma^n_i)$ are called \emph{degeneracies} and denoted $s^n_i$. A \emph{demi-simplicial set}\footnote{It seems these objects have not been considered previously in the literature. Since simplicial sets without degeneracies are called \emph{semi-simplicial sets}, we have chosen to replace the prefix \emph{semi-} by its French version \emph{demi-}. Additionally, the first letter of each prefix specifies whether we are removing degeneracies or face maps.} $X$ is a functor $X\colon\Delta^\op_\surj\to \Set$.
\end{defn}

Faces and degeneracies satisfy the so-called \emph{simplicial identities} which are the dual of the cosimplicial identities mentioned above. In the particular case of the relation involving only degeneracies, we get the identity
\begin{equation}\label{eq:deg_rel}
\sigma_j\circ \sigma_i = \sigma_{j-1}\circ \sigma_i \qquad i\leq j
\end{equation}

In what follows, when we refer to conditions in Roman numerals, we mean the conditions satisfied by cloning systems and bilateral cloning systems from Definition~\ref{defn:CloningSystem} and Definition~\ref{defn:NCCloningSystem}.

Observe that \eqref{eq:deg_rel} is exactly the same relation satisfied by the maps $\kappa^n_j$ with the extra condition (iv+) except that the subindices are shifted by one (the first map is $\kappa_1$ not $\kappa_0$). Therefore we have the following consequence. 
\begin{lem} A demi-simplicial set with values in groups is the same as a family of groups 
 $\G_{\bullet}=\{\G_n\}_{n\geq 1}$  equipped with  maps $\kappa=\{\kappa^n_{j}\colon\G_{n}\to\G_{n+1}\}_{n\ge 1,\,1\leq j\leq n}$ satisfying condition \emph{(iv+)}. $\hfill\qed$
\end{lem}
\subsection{The demi-simplicial set of symmetric groups} Let us build the demi-simplicial set associated to the cloning system of symmetric groups. Note that an order-preserving surjection $f\colon [n]\to [m]$ is completely determined by the cardinality of the preimages $f^{-1}(0),\ldots,f^{-1}(m)$. Every permutation $h\in \Sigma_{m+1}$ of $[m]$ induces a permutation of the preimages $f^{-1}(0),\ldots,f^{-1}(m)$, and therefore a block permutation on $[n]$, that we denote by $\Phi(f)(h)\in \Sigma_{n+1}$. This defines a functor $\Phi\colon \Delta_\surj^\op\to \Set$ with $\Phi([n]) = \Sigma_{n+1}$.

\begin{defn}\label{defn: Sigma acting on Delta surj} Define an action 
 $$
    \begin{tikzcd}[ampersand replacement=\&]
    \Sigma_{m+1}\times \Delta_\surj([n],[m])\ar[rr] \&\& \Delta_{\surj}([n],[m]) \\[-6mm]
    (g,f) \ar[rr, mapsto] \&\& f_g
    \end{tikzcd},
 $$
where $f_g\colon [n]\to [m]$ is the unique order-preserving surjection such that the cardinality of ${f_g}^{-1}(g(i))$ is equal to the cardinality of $f^{-1}(i)$. 
\end{defn}
Observe now that the map $\Phi(f)\colon \Sigma_{m+1}\to \Sigma_{n+1}$ is not a group homomorphism but satisfies that 
\begin{equation*}\label{eq:610}
    \Phi(f)(g\cdot g') = \Phi(f_{g'})(g)\cdot\Phi(f)(g'),
\end{equation*}
which applied to degeneracies is the same as condition (v) in a cloning system
\begin{equation}\label{eq:611}\Phi(\sigma_i)(g\cdot g') = \Phi(\sigma_{g'(i)})(g)\cdot \Phi(\sigma_i)(g').
\end{equation}
The functor $\Phi$, defined on order-presering surjections between ordinals, can be extended to any order-preserving map between ordinals. Indeed, since every map factors uniquely as an order-preserving surjection followed by an order-preserving injection, it is enough to define it on injections, which we can do as follows. Note that to give an order-preserving injective function $f\colon [n]\to [m]$ is equivalent to specify the complement of the image $A = [m]\smallsetminus f([n])$.

\begin{defn}\label{defn: Sigma acting on Delta inj} Define an action 
 $$
    \begin{tikzcd}[ampersand replacement=\&]
    \Sigma_{m+1}\times \Delta_\inj([n],[m])\ar[rr] \&\& \Delta_{\inj}([n],[m]) \\[-6mm]
    (g,f) \ar[rr, mapsto] \&\& f_g
    \end{tikzcd},
 $$
 where $f_g\colon [n]\to [m]$ is the unique order-preserving injection for which the following holds $[m]\smallsetminus f_g([n]) = [m]\smallsetminus g(f([m]))$. 
\end{defn}

Observe that, by the unique factorization of any map in $\Delta$ mentioned above, the actions given in Definitions \ref{defn: Sigma acting on Delta surj}, \ref{defn: Sigma acting on Delta inj} combine to yield an action
 $$
    \begin{tikzcd}[ampersand replacement=\&]
    \Sigma_{m+1}\times \Delta([n],[m])\ar[rr] \&\& \Delta([n],[m]) \\[-6mm]
    (g,f) \ar[rr, mapsto] \&\& f_g
    \end{tikzcd}.
 $$
Thus, for a permutation $g\in \Sigma_{m+1}$  of $[m]$, we set $\Phi(f)(g) = f_g^{-1}\circ g\circ f$. This defines a functor $\Phi\colon \Delta^\op\to \Set$, and hence a simplicial set. Observe that if $f\colon [n]\to [m]$ is any map, then the map $\Phi(f)\colon \Sigma_{m+1}\to \Sigma_{n+1}$ is not a group homomorphism but satisfies  
\begin{equation*}\label{eq:613}
    \Phi(f)(g\cdot g') = \Phi(f_{g'})(g)\cdot\Phi(f)(g').
\end{equation*}

\subsection{Crossed simplicial groups and crossed interval groups}

The following definition of crossed simplicial groups is a characterization taken from~\cite[Proposition 1.7]{FL}, where $\Phi$ denotes the simplicial set constructed above.
\begin{defn}
    A \emph{crossed simplicial group} is a simplicial set $\Psi\colon \Delta^\op\to \Set$ with objectwise values on groups together with a levelwise group morphism $\pi\colon \Psi\to \Phi$ such that 
    \[
    \Big(\big(\pi\circ \Psi(\sigma_i)\big)(g)\Big)(j) = \Big(\big(\Phi(\sigma_i)\circ \pi\big)(g)\Big)(j)
    \]
    for all $j\neq i,i+1$ and
\begin{align*}
    \Psi(\sigma_i)(g\cdot g') &= \Psi(\sigma_{\pi(g')(i)})(g)\cdot\Psi(\sigma_i)(g'), \\ 
    \Psi(\delta_i)(g\cdot g') &= \Psi(\delta_{\pi(g')(i)})(g)\cdot\Psi(\delta_i)(g').
    \end{align*} 
    A \emph{crossed demi-simplicial group} is defined in the same way, replacing the category $\Delta$ by the category $\Delta_\surj$. 
\end{defn}

The following result follows immediately from the previous discussion and the definition of crossed demi-simplicial group.
\begin{lem} Let $\G_{\bullet}=\{\G_n\}_{n\geq 1}$ be a family of groups, $\pi=\{\pi_n\colon \G_n\to\Sigma_n\}_{n\ge 1}$ a family of group homomorphisms and $\kappa=\{\kappa^n_{j}\colon\G_{n}\to\G_{n+1}\}_{n\ge 1,\,1\leq j\leq n}$ a family of maps. A triple $(\G_\bullet,\pi,\kappa)$ satisfying conditions {\rm (ii)}, {\rm (iv+)} and {\rm (v)} is the same as a crossed demi-simplicial group. $\hfill\qed$
\end{lem}

Now we would like to incorporate the homomorphisms $\iota$ into the picture. We will incorporate the morphisms $\zeta$ at the same time.

\begin{defn} For each $n\geq 1$, the $n^{\text{th}}$ \emph{interval} is the set $\langle n\rangle = \{-\infty,1,\ldots,n,\infty\}$. The \emph{interval category} $\I$ is the category whose objects are all intervals and whose morphisms are order-preserving maps that preserve $-\infty$ and $\infty$. The subcategory $\I_\surj$ has the same objects as $\I$ and its morphisms are the order-preserving surjective maps that preserve $-\infty$ and $\infty$.\end{defn}
The interval category can be introduced as the Joyal dual of the simplicial category \cite{Joyal} or as the image of the faithful embedding $\alpha\colon \I\to \Delta$ that sends $\langle n\rangle = \{-\infty,1,\ldots,n,\infty\}$ to $[n+1] = \{0,1,\ldots,n+1\}$, and an interval map yields a simplicial map by interpreting $-\infty$ as $0$ and $\infty$ as $n+1$. Here we are interested in the second description, and we will blur the difference between maps in $\I$ and their images under the embedding $\alpha$. Therefore, from now on we will write simply $\Phi$ for the crossed interval group $\Phi\circ \alpha$.
\begin{defn}
    An \emph{inert crossed interval group} is a presheaf $\Psi\colon \I^\op\to \Set$ with values on groups together with a levelwise group morphism $\pi\colon \Psi\to \Phi$ such that 
    \[
    \Big(\big(\pi\circ \Psi(\sigma_i)\big)(g)\Big)(j) = \Big(\big(\Phi(\sigma_i)\circ \pi\big)(g)\Big)(j)
    \]
    for all $j\neq i,i+1$ and
    \begin{align*}\label{eq:product_cig}
    \Psi(\sigma_i)(g\cdot g') &= \Psi(\sigma_{\pi(g')(i)})(g)\cdot\Psi(\sigma_i)(g'), \\ 
    \Psi(\delta_i)(g\cdot g') &= \Psi(\delta_{\pi(g')(i)})(g)\cdot\Psi(\delta_i)(g').
    \end{align*} 
     An \emph{inert crossed demi-interval group} is defined in the same way, by replacing $\I$ by $\I_\surj$.
\end{defn}
Observe now that if $\Phi$ is an inert crossed demi-interval group, the maps $\Phi(s_0)$ and $\Phi(s_{n+1})$ are group homomorphisms. In fact, we can interpret bilateral cloning systems as inert crossed demi-interval groups by setting $\kappa^n_i = s^{n+1}_i$, $\zeta_n = s^{n+1}_0$ and $\iota_n = s^{n+1}_{n+1}$. Thus, we have the following result.
\begin{prop}\label{prop: Bilateral CS to inert crossed demi-interval gp} A quintuple $(\G_\bullet,\pi,\iota,\zeta,\kappa)$ satisfying all the conditions of a bilateral cloning system is the same as an inert crossed demi-interval group. 
\end{prop}

\begin{rem}
    Every action operad with constants gives rise to a crossed simplicial group. This is carefully developed in \cite[2.4]{Zhang}. This construction is different from ours, and does not allow to recover the composition products of the operad structure from the simplicial structure.
\end{rem}

\begin{rem} Crossed interval groups have been studied in \cite{Batanin-Markl} and \cite{Yoshida}. The adjective \emph{inert} corresponds to any of the two equivalent properties defined in \cite[Lemma 4.3]{Yoshida}. From our viewpoint, the adjective \emph{inert} implies that the composite $\pi_n\circ \Psi(\sigma_i)$ is the identity permutation when $i=0, n$, and therefore $\Psi(\sigma_i)$ is a group homomorphism in those cases. This condition is automatic from our definition of the crossed interval group $\Phi$. 

In that paper, Yoshida studied the relation between crossed interval groups and action operads with constants (operations of arity 0). He established that every action operad with constants determines a crossed interval group, and found three properties that characterise the crossed interval groups that come from an action operad with constants: operadicness, tameness and ``factoring through the symmetric group''. Under the above lemma, tameness corresponds to property (ix), while operadicness corresponds to property (viii) plus being inert and ``factoring through the symmetric group'' corresponds to (ii+). Dropping the constants, one obtains the vertical maps of the diagram of the introduction. 

We note that Examples 3.1 and 3.4 in \cite{Zaremsky}, which do not come from an action operad, do come from an inert crossed demi-interval group (with trivial homomorphisms $\pi$). 
\end{rem}

\begin{rem}\label{rem:crossed_injectivity}
    In their initial definition of cloning system, Witzel and Zaremsky request the homomorphisms $\iota$ to be injective. This is not necessary for their constructions, but it is often satisfied in practice. We think the following is a possible reason for this: In a crossed interval group $\Psi$ the maps $\Psi(\sigma_i)$ are always injective by the simplicial identity $\sigma_i\circ \delta_i = \id$. Similarly, in an action operad \emph{with constants} the maps $g\mapsto g\circ_i e_2$ and $g\mapsto e_2\circ_i g$ are injective, because the maps $g\mapsto g\circ_i e_0$ and $g\mapsto e_0\circ_i g$ are left inverses of them. 
    
    Most of the operadic cloning systems considered correspond to action operads without constants that can be obtained by forgetting the constants from an action operad with constants. For example, the action operad of braid groups comes from an action operad with constants: The group of constants $\Br_0$ is the trivial group and if $g\in \Br_{n}$ is a braid, then $g \circ_i e_0\in \Br_{n-1}$ is the braid obtained from $g$ by forgetting the $i$-th strand. 
\end{rem}


\section{Cloning systems and PROs}\label{section:props}

In this section, we present another viewpoint for cloning systems in terms of PROs (product categories) that will be used in a forthcoming piece of work about the construction of Thompson groups. We begin by recalling the definition of PRO.
\begin{defn} A \emph{PRO} $\Ofrak$ is a quadruple $(\Ofrak,\odot,\otimes,\id)$ where:
\begin{itemize} 
    \item $\Ofrak$ is a collection of sets $\Big(\Ofrak\brbinom{n}{m}\Big)_{n,m\geq 1}$,
    \item $\odot$ is an associative product (\emph{vertical product}) with two-sided unit $\id$
    $$
    \odot\colon \Ofrak\brbinom{n}{t}\times \Ofrak\brbinom{s}{n}\longrightarrow \Ofrak\brbinom{s}{t},\quad  \id_{n}\in\Ofrak\brbinom{n}{n},
    $$
    \item $\otimes$ is an associative product (\emph{horizontal product})
    $$
    \otimes\colon \Ofrak\brbinom{n_1}{m_1}\times \Ofrak\brbinom{n_2}{m_2}\longrightarrow \Ofrak\brbinom{n_1+n_2}{m_1+m_2}.
    $$
\end{itemize}
Moreover, the vertical and horizontal product are required to satisfy the \emph{interchange law}
$$
(f\otimes g)\odot (p\otimes q)=(f\odot p)\otimes (g\odot q),
$$
whenever it makes sense, and  $\id_{n}\otimes\id_{m}=\id_{n+m}$ for all $n,m\geq 1$.
\end{defn}

More abstractly, a \emph{PRO} $\mathfrak{O}$ is a strict (non unital) monoidal category equipped with a strict monoidal functor $(\mathbb{N}_{\geq 1},+)\to (\mathfrak{O},\otimes)$ which is an isomorphism on objects. We do not consider monoidal units, that is, units for $\otimes$, because they will create nullary operations that we are avoiding.

\begin{examp}[Symmetric groups]
The family of symmetric groups $\{\Sigma_n\}_{n\ge 1}$ yields a very simple PRO by setting
$$
\Sigma\brbinom{n}{m}= \begin{cases}
    \Sigma_n & \mbox{if $n=m$}, \\
    \emptyset & \mbox{if $n\neq m$}.
\end{cases}
$$
The vertical product $\odot$ is just the group structure of the symmetric groups and the horizontal product $\otimes$ is the block product of permutations (see Figure ~\ref{fig:HorizontalProductSigma}).
\begin{figure}[htp]
    \centering
    \includegraphics[width=6cm]{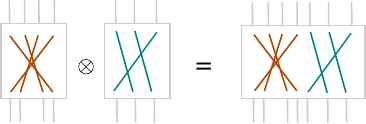}
    \caption{Block product of $(1\;4)(2\;3)$ and $(1\;2\;3)$.}
    \label{fig:HorizontalProductSigma}
    \end{figure}\\
We denote by $\Sigma$ the PRO of symmetric groups. Additionally, we consider that $\Sigma$ is equipped with \emph{cloning maps} for any $n$-tuple of positive integers $(m_1,\dots,m_n)$, 
$
c(m_1,\dots,m_n)\colon \Sigma\brbinom{n}{n}\to \Sigma\brbinom{m_1+\dots+m_n}{m_1+\dots+m_n},
$ 
obtained by replacing the $i$-th strand by $m_i$ strands for all $1\leq i\leq n$.
\end{examp}

\begin{examp}[Signed symmetric groups]
The family of signed symmetric groups $\{\Sigma_{n}^{\pm}\}_{n\ge 1}$ provides a slightly more complicated PRO. First, define 
$$
\Sigma^{\pm}\brbinom{n}{m}= \begin{cases}
    \Sigma_{n}^{\pm} & \mbox{if $n=m$}, \\
    \emptyset & \mbox{if $n\neq m$}.
\end{cases}
$$
The vertical product $\odot$ is just the group structure of the signed symmetric groups and the horizontal product $\otimes$ is the block product of signed permutations, which is defined analogously to the block product of ordinary permutations.
We denote by $\Sigma^{\pm}$ the PRO of signed symmetric groups. Additionally, we consider that $\Sigma^{\pm}$ is equipped with \emph{cloning maps} for any $n$-tuple of positive integers $(m_1,\dots,m_n)$, 
$
c(m_1,\dots,m_n)\colon \Sigma^{\pm}\brbinom{n}{n}\to \Sigma^{\pm}\brbinom{m_1+\dots+m_n}{m_1+\dots+m_n},
$ 
obtained by ``plumbing'' $m_i$ strands in the $i$-th component for $1\leq i\leq n$ subject to the rules: given $g\in \Sigma^{\pm}$, 
\begin{itemize}
\item if $g(i)$ is positive, the new $m_i$ strands in $c(m_1,\dots,m_n)(g)$ are plumbed with the identity permutation;
\item if $g(i)$ is negative, the new $m_i$ strands are plumbed with the permutation 
$$
\left(\begin{matrix}
1 & 2 & \cdots & m_i-1 & m_i\\
m_i & m_i-1 & \cdots & 2 & 1
\end{matrix}\right).
$$
\end{itemize}
\begin{figure}[h]
    \centering
    \includegraphics{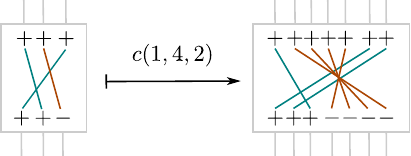}
    \caption{Cloning map on $\Sigma^{\pm}$.}
    \label{fig:KappaSigned}
    \end{figure}
\end{examp}

We will also need the concept of a \emph{morphism of PROs} $f\colon \Ofrak\to\Ofrak'$, which consists of maps $f_{n,m}\colon\Ofrak\brbinom{n}{m}\to \Ofrak'\brbinom{n}{m}$ for every $n,m\ge 1$ that preserve the PRO structure.

Next, we define two special types of PROs, by adding some extra structure, which will turn out to be equivalent to operadic cloning systems and their restricted counterparts. 

\begin{defn}
A \textit{cloning PRO} consists of the following data:
\begin{itemize}
    \item a PRO $(\G,\odot,\otimes,\id)$ such that the vertical product on $\G\brbinom{n}{n}$ admits inverses, that is, $\G\brbinom{n}{n}$ is a group for every $n$, and $\G\brbinom{n}{m}=\emptyset$ if $n\neq m$, 
    \item a morphism of PROs $\pi\colon\G\to\Sigma^{\pm}$,
    \item cloning maps 
    $\kappa(\underline{m})\colon \G\brbinom{n}{n}\rightarrow\G\brbinom{m_1+\dots+m_n}{m_1+\dots+m_n}$ for every $n\geq 1$ and every $n$-tuple $\underline{m}=(m_1,\dots,m_n)$ with $m_i\geq 1$, satisfying $\kappa(1,\stackrel{(n)}{\dots}, 1)=\id_{\G_n}$, and  the associativity condition
    $$
    \kappa(\underline{r}^1\sqcup\cdots\sqcup\underline{r}^{n})\circ\kappa(\underline{m})=\kappa\Big(\sum_{j_1}r^1_{j_1},\dots,\sum_{j_n}r^n_{j_n}\Big) 
    ,
    $$
    where $\underline{r}^i=(r^i_1,\dots,r^i_{m_i})$ and  
    $
    \underline{r}^1\sqcup\cdots\sqcup\underline{r}^n=(r^1_1,\dots,r^1_{m_1},\dots,r^n_1,\dots,r^n_{m_n}).
    $
\end{itemize}
Moreover, the PRO structure is compatible with $\pi$ and $\kappa$ as stated by the following conditions:
\begin{enumerate}
    \item $\pi$ commutes with  $\kappa$: $\pi\circ\kappa(\underline{m})=c(\underline{m})\circ\pi$.
     \item Identities and $\kappa$: $\kappa(\underline{m})(\id_{n})=\id_{m_1+\cdots+m_n}$.
    \item Horizontal composition and $\kappa$:
    $$\kappa(\underline{m})(f)\otimes \kappa(\underline{n})(g)= \kappa(\underline{m}\sqcup\underline{n})(f\otimes g).$$
    \item Vertical composition and $\kappa$:
    $$\kappa(\underline{m})(h\odot g)=\kappa(\underline{m}_{\pi(g)})(h)\odot \kappa(\underline{m})(g),$$ where  $\underline{m}_{\pi(g)}=(m_{\vert\pi(g)(1)\vert},\dots,m_{\vert\pi(g)(n)\vert})$. 
    \item[($\star$)] Twisted interchange law: $$\kappa(\underline{m})(f)\odot (g_{1}\otimes \cdots\otimes g_{n})=(g_{\vert\pi(f)(1)\vert}\otimes \cdots \otimes g_{\vert\pi(f)(n)\vert})\odot \kappa(\underline{m})(f).$$
\end{enumerate}

A \emph{morphism of cloning PROs} $f\colon (\G,\odot,\otimes,\id,\pi,\kappa)\to (\G',\odot',\otimes',\id',\pi',\kappa') $ is a morphism of PROs over $\Sigma^{\pm}$ which is compatible with the cloning maps, i.e.\ $f\circ \kappa(\underline{m})=\kappa(\underline{m})\circ f$ for any $\underline{m}$.
\end{defn}
\begin{defn}
    A cloning PRO $(\G,\odot,\otimes, \id,\pi,\kappa)$ is said to be \emph{restricted} if the morphism of cloning PROs $\pi\colon \G\to \Sigma^{\pm}$ factors through the cloning PRO of symmetric groups $\Sigma$, which sits inside $\Sigma^{\pm}$ as positive permutations.
\end{defn}

After these definitions, we prove that every (restricted) operadic cloning system gives rise to a (restricted) cloning PRO. More concretely: 

\begin{prop}\label{prop:NCCStoPRO} A (restricted) operadic cloning system $(\G_{\bullet},\iota,\zeta,\kappa,\pi)$ defines a (restricted) cloning PRO  $(\G,\odot,\otimes,\id,\pi,\kappa)$ in a functorial way.
\end{prop}
\begin{proof}
Set $\G\brbinom{n}{n}=\G_n$, the vertical product $\odot$ to be the group multiplication and $\id_n=e_n$. Define the cloning maps
$$
\kappa(\underline{m})=\kappa(m_1,\dots,m_n)=\kappa_1(m_1)\circ\cdots \circ \kappa_{n}(m_n),
$$
and the horizontal product $\otimes$
$$
\begin{tikzcd}
\otimes\colon \G\brbinom{n}{n}\times \G\brbinom{m}{m}\ar[rr,"\iota(m)\times \zeta(n)"] && \G\brbinom{n+m}{n+m}\times\G\brbinom{n+m}{n+m}\ar[r, "\cdot"] & \G\brbinom{n+m}{n+m} 
\end{tikzcd}.
$$

We now check that $(\G,\odot,\otimes,\id,\pi,\kappa)$ is a (restricted) cloning PRO by verifying all the axioms. Let $f\in\G_n$, $g\in\G_m$ and $h\in \G_r$.
\begin{itemize}
    \item Associativity of $\otimes$: 
\begin{align*}
(f\otimes g)\otimes h &=  \iota(r)\big(\iota(m)(f)\cdot \zeta(n)(g)\big)\cdot \zeta(n+m)(h)\\
&= \iota(r+m)(f)\cdot \iota(r)(\zeta(n)(g))\cdot \zeta(n+m)(h)\\
&= f\otimes (g\otimes h),
\end{align*}
by condition (vi) of bilateral cloning systems and since $\iota$ and $\zeta$ are homomorphisms. Also, $\id_n\otimes\id_m=\id_{n+m}$ since $\iota(m)$ and $\zeta(n)$ preserve identities and $e_{n+m}$ is idempotent.
\item Interchange law for $\odot$ and $\otimes$: 
\begin{align*}
(f\otimes g)\odot (p\otimes q) &=  \Big(\iota(m)(f)\cdot \zeta(n)(g)\Big)\cdot\Big(\iota(m)(p)\cdot\zeta(n)(q)\Big)\\
&= \iota(m)(f)\cdot \iota(m)(p)\cdot \zeta(n)(g)\cdot\zeta(n)(q)\\
&= (f\odot p)\otimes (g\odot q),
\end{align*}
by condition (ix) and since $\iota$ and $\zeta$ are homomorphisms.

\item $\pi\colon \G\to \Sigma^{\pm}$ (resp.\ $\pi\colon \G\to \Sigma$) is a morphism of PROs, since $\pi$ commutes with $\iota,\ \zeta$ and it consists of group homomorphisms (see Remark \ref{rem: Condition ii' for cloning systems}).

\item Associativity conditions for $\kappa$ hold by conditions (iv) and (iv+) (axioms for compositions of $\kappa$'s for operadic/bilateral cloning systems).

\item Commutation of $\kappa$ and $\pi$ holds by the analogous axiom for (restricted) operadic cloning systems, i.e.\ (ii') (resp.\ (ii+)).

\item $\kappa$ acting on identities. Since
$
\kappa_j(r)(e_m)=\kappa_j(r)(e_m)\cdot\kappa_j(r)(e_m)
$
(see the observation after (\ref{eqt:6.5})),
it follows that $\kappa_j(r)(e_m)=e_{m+r-1}$.

\item Relation between $\kappa$ and $\odot$. By iterating condition (v) we get 
\begin{multline*}
\kappa(\underline{m})(f\odot p) =  \kappa_1(m_1)\circ\cdots\circ \kappa_n(m_n)(f\cdot p)\\
=  \Big(\kappa_{p(1)}(m_1)\circ\cdots\circ\kappa_{p(n)}(m_n)(f)\Big)\cdot \Big(\kappa_1(m_1)\circ\cdots\circ\kappa_n(m_n)(p)\Big)\\
= \kappa(\underline{m}_{p})(f)\odot \kappa(\underline{m})(p),
\end{multline*}
where we have adopted Notation \ref{notat: Dropping pi when understood}.

\item Relation between $\kappa$ and $\otimes$. Let us analyze what happens for $\kappa_j(r)$ instead of a general $\kappa(\underline{m}^1\sqcup\underline{m}^2)$, because the general case will follow by combining these simple cases by conditions (iv) and (iv+). 

\begin{itemize}
    \item[\textbf{Case 1:}] ($1\leq j\leq n$). 
    \begin{align*}
 \kappa_j(r)(f\otimes g) &=  \kappa_j(r)\big(\iota(m)(f)\cdot\zeta(n)(g)\big)\\
 &=  \kappa_{\zeta(n)(g)(j)}(r)(\iota(m)(f))\cdot \kappa_j(r)(\zeta(n)(g))\\
 &=\kappa_j(r)(\iota(m)(f))\cdot \kappa_j(r)(\zeta(n)(g))\\
 &=\iota(m)\big(\kappa_j(r)(f)\big)\cdot \zeta(r+n)(g)\\
 &= \kappa_j(r)(f)\otimes g.
 \end{align*}
We have applied conditions (v), (iii), (vii') and that $\zeta(n)(g)$ is constant over $\{1,\dots, n\}$.
    \item[\textbf{Case 2:}] ($n<j\leq m+n$).
    \begin{align*}
 \kappa_j(r)(f\otimes g) 
 &=  \kappa_{n+g(j)}(r)(\iota(m)(f))\cdot \kappa_j(r)(\zeta(n)(g))\\
 &=\iota(m+r)(f)\cdot \zeta(n)(\kappa_j(r)(g))\\
 &= f\otimes \kappa_j(r)(g).
 \end{align*}
We have applied conditions (v), (iii'), (vii) and that $\zeta(n)(g)(j)$ in this case is $n+g(j)$ which is strictly bigger than $n$ (recall that we are always assuming that elements in $\G_m$ or $\Sigma^{\pm}_m$ act on $\{1,\dots, m\}$ through their canonical map into $\Sigma_m$).
\end{itemize}

\item Twisted interchange law. Let us just check the case $f\in \G_2$ ($n=2$) and $g_i\in \G_{m_i}$, since the general case follows from the same ideas.
 \begin{align*}
\kappa(m_1,m_2)(f)\odot (g_1\otimes g_2) 
 &= (\kappa_1(m_1)\circ\kappa_2(m_2))(f)\cdot \Big(\iota(m_2)(g_1)\cdot\zeta(m_1)(g_2)\Big) \\
 &\overset{(a)}{=} \Big(\kappa_1(m_1)\big(\kappa_2(m_2)(f)\big)\cdot \nu_{1}(m_2)(g_1)\Big)\cdot \zeta(m_1)(g_2)\\
 &\overset{(b)}{=} \Big(\nu_{f(1)}(m_2)(g_1)\cdot \kappa_{1}(m_1)\big(\kappa_{2}(m_2)(f)\big)\Big)\cdot \zeta(m_1)(g_2)\\
  &\overset{(c)}{=} \nu_{f(1)}(m_2)(g_1)\cdot \kappa_{1}(m_1)\Big(\kappa_{2}(m_2)(f)\cdot \nu_2(1)(g_2)\Big)\\
  &\overset{(d)}{=} \nu_{f(1)}(m_2)(g_1)\cdot \kappa_{1}(m_1)\Big(\nu_{f(2)}(1)(g_2)\cdot\kappa_{2}(m_2)(f)\Big)\\
   &\overset{(e)}{=} \nu_{f(1)}(m_2)(g_1)\cdot \nu_{f(2)}(m_1)(g_2)\cdot\kappa(m_1,m_2)(f)\\
 &\overset{\phantom{(e)}}{=} (g_{f(1)}\otimes g_{f(2)})\odot \kappa(m_1,m_2)(f).
 \end{align*}
In the above equalities, we have applied: $(a)$ definition of $\nu$; $(b)$ condition (viii); $(c)$ conditions (v), (vii') and definition of $\nu$; $(d)$ condition (viii); and $(e)$ conditions (v), (vii) and (iv+). \hfill\qedhere
\end{itemize}
\end{proof}

Next, we prove that (restricted) operadic cloning systems come from (restricted) cloning PROs: 

\begin{prop}\label{prop:PROsToNCCS} A (restricted) cloning PRO $(\G,\odot,\otimes,\id,\pi,\kappa)$ determines a (restricted) operadic cloning system $(\G_{\bullet},\iota,\zeta,\kappa,\pi)$ in a functorial way.
\end{prop}
\begin{proof}
Set $\G_n=\G\brbinom{n}{n}$, the product $\cdot=\odot$ and $e_n=\id_n$. Define
$$
\kappa_j(m)=\kappa(1,\dots,1,m^{(j)},1,\dots,1), \quad 
\iota(f)=f\otimes \id \quad \text{and} \quad \zeta(f)=\id\otimes f.
$$

The only non-straightforward axioms to check are the following:
\begin{itemize}
    \item $\iota$ is a homomorphism of groups. We have that,
    $$
    \iota(f\cdot g)= (f\odot g)\otimes \id_1= (f\otimes \id_1)\odot(g\otimes \id_1)= \iota(f)\cdot\iota(g)
    $$
    by the interchange law and since $\id_1$ is idempotent. A similar argument proves that $\zeta$ is also a homomorphism of groups.
    \item Axioms (iii) and (vii). We have that
    $$
    \kappa_j(r)\big(\iota(m)(f)\big)=\left\lbrace 
    \begin{array}{lcl}
    \kappa_j(r)(f)\otimes \id_m=\iota(m)\big(\kappa_j(r)(f)\big) && \text{if }1\leq j\leq n,\\[4mm]
   f\otimes \id_{m+r-1}=\iota(m+r-1)(f) && \text{otherwise},
    \end{array}
    \right.
    $$
    where $f\in\G\brbinom{n}{n}$. The corresponding conditions (iii') and (vii') for $\zeta$ can be deduced in the same way.
    \item Axiom (vi). We have that 
    $$
    \zeta(m)\big(\iota(n)(f)\big)=\id_m\otimes f\otimes \id_n= \iota(n)\big(\zeta(m)(f)\big)
    $$
    by the associativity of the horizontal product $\otimes$.
    \item Axiom (viii). We have that
    \begin{align*}
 \kappa_j(m)(f)\cdot \nu_j(n)(g)&= \kappa(1,\dots,m^{(j)},\dots, 1)(f)\odot(\id_{j-1}\otimes\, g\otimes \id_{n-j+1})\\
 &= (\id_{f(j)-1}\otimes\, g\otimes \id_{n-f(j)+1})\odot \kappa(1,\dots,m^{(j)},\dots, 1)(f)\\ 
 &=\nu_{f(j)}(n)(g)\cdot \kappa_j(m)(f),
 \end{align*}
 by the twisted interchange law and the identity $\id_{r}\otimes\id_{s}=\id_{r+s}$.
    \item Axiom (ix). We have that
     \begin{align*}
 \iota(m)(f)\cdot \zeta(n)(g)&= (f\otimes \id_m)\odot(\id_n\otimes\, g)\\
 &= (f\odot \id_n)\otimes (\id_m\odot\, g)\\ 
 &= (\id_n\odot\, f)\otimes (g\odot \id_m)\\
  &= (\id_n\otimes\, g)\odot (f\otimes \id_m)\\
 &=\zeta(n)(g)\cdot \iota(m)(f),
 \end{align*}
    by the interchange law and since $\id_r$ is an identity for $(\G\brbinom{r}{r},\odot)$. \hfill\qedhere
\end{itemize}
\end{proof}

Finally, we have the following theorem, which ties together the results in this section:

\begin{thm}\label{thm:Bijection PROs and CSs} The functorial constructions of Propositions \ref{prop:NCCStoPRO} and \ref{prop:PROsToNCCS} induce isomorphisms of categories
$$
\begin{Bmatrix}
\text{(restricted) operadic}\\
\text{cloning systems}
\end{Bmatrix} \cong 
\begin{Bmatrix}
\text{(restricted)}\\
\text{cloning PROs}
\end{Bmatrix}.
$$
\end{thm}
\begin{proof}
Let us first check that
$$
\begin{Bmatrix}
\text{(restricted) operadic}\\
\text{cloning systems}
\end{Bmatrix} \to 
\begin{Bmatrix}
\text{(restricted)}\\
\text{cloning PROs}
\end{Bmatrix}\to 
\begin{Bmatrix}
\text{(restricted) operadic}\\
\text{cloning systems}
\end{Bmatrix}.
$$
is the identity. Let $(\G_{\bullet},\iota,\zeta,\kappa,\pi)$ be a (restricted) operadic cloning system, and let $(\hat{\G}_{\bullet},\hat{\iota},\hat{\zeta},\hat{\kappa},\hat{\pi})$ be the image of the previous (restricted) operadic cloning system under the composition of functors. By definition $\hat{\G}=\G$ as groups and $\hat{\pi}=\pi$. For $\hat{\iota}$, we have that
\begin{align*}
\hat{\iota}(g) &= g\otimes \id = \iota(1)(g)\cdot \zeta(n)(e_1)=\iota(g)\cdot e_{n+1}=\iota(g).
\end{align*}
The maps $\zeta$ behave similarly, and
\begin{align*}
\hat{\kappa}_j(g) &= \kappa(1,\dots,2^{(j)},\dots,1)(g)=\big(\kappa_1(1)\circ\cdots\circ\kappa_j(2)\circ\cdots\circ\kappa_n(1)\big)(g)=\kappa_j(g).
\end{align*}

Conversely, let $(\G,\odot,\otimes,\id,\pi,\kappa)$ be a (restricted) cloning PRO and let the tuple  $(\hat{\G},\hat{\odot},\hat{\otimes},\hat{\id},\hat{\pi},\hat{\kappa})$ denote the image of it under the other composition of functors. It is clear that $\hat{\G} = \G$, $\hat{\odot}=\odot$, $\hat{\id}=\id$, $\hat{\pi}=\pi$ and $\hat{\kappa}=\kappa$. Finally, regarding the horizontal product we have that
\begin{align*}
f\,\hat{\otimes}\, g &= \iota(m)(f)\cdot \zeta(n)(g)= (f\otimes \id_m)\odot (\id_n\otimes g)\\ 
&= (f\odot \id_n)\otimes(\id_m \odot g) = f\otimes g. \qedhere
\end{align*}
\end{proof}

\bibliographystyle{aomalpha}
\bibliography{references.bib}

\end{document}